\documentclass[11pt,twoside]{article}
\usepackage[T1]{fontenc}
\usepackage{amsmath,amsfonts,amssymb,amsthm,fullpage,bbm}
\usepackage{fancyhdr,graphicx,color}
\usepackage{epsfig}
\usepackage{yhmath}

\newtheorem{thm}{Theorem}
\newtheorem{lem}{Lemma}

\newtheorem{prop}{Proposition}
\newtheorem{rem}{Remark}

\newcommand{\eps}{\varepsilon}

\def\PP{\mathbb{P}}
\def\RR{\mathbb{R}}
\def\EE{\mathbb{E}}
\def\NN{\mathbb{N}}
\def\ZZ{\mathbb{Z}}
\def\QQ{\mathbb{Q}}

\def\wT{\widetilde{T}}
\def\wB{\widetilde{B}}
\def\wx{\widetilde{x}}
\def\wy{\widetilde{y}}
\def\w0{\widetilde{0}}
\def\wmu{\widetilde{\mu}}
\def\L{\mathcal{L}}
\def\B{\mathcal{B}}
\def\C{\mathcal{C}}

\def\ind{{\mathbbm{1}}_}

\title{\huge Weak shape theorem in first passage percolation\\ with infinite passage times}
\author{}
\date{}

\begin{document}
\maketitle

\thispagestyle{empty}

\begin{center}
\vskip-1cm {\Large Rapha\"el Cerf}\\
{\it DMA, Ecole Normale Supérieure\\
45 rue d'Ulm, 75230 Paris Cedex 05, France}\\
{\it E-mail:} rcerf@math.u-psud.fr\\
\vskip0.5cm and\\
\vskip0.5cm {\Large Marie Th\'eret}\\
{\it LPMA, Universit\'e Paris Diderot\\ 
Bâtiment Sophie Germain, 5 rue Thomas Mann, 75205 Paris Cedex 13, France}\\
{\it E-mail:} marie.theret@univ-paris-diderot.fr
\end{center}

\noindent
{\bf Abstract:}
We consider the model of i.i.d. first passage percolation on $\ZZ^d$ : we associate with each edge $e$ of the graph a passage time $t(e)$ taking values in $[0,+\infty]$, such that $\PP[t(e)<+\infty] >p_c(d)$. Equivalently, we consider a standard (finite) i.i.d. first passage percolation model on a super-critical Bernoulli percolation performed independently. We prove a weak shape theorem without any moment assumption. We also prove that the corresponding time constant is positive if and only if $\PP[t(e)=0]<p_c(d)$.\\





\section{Introduction}

\subsection{First definitions and main results}

We consider the standard model of first passage percolation on $\ZZ^d$ for $d\geq 2$. Let $\EE^d$ be the set of the edges $e=\langle x,y \rangle$ of endpoints $x = (x_1,\dots ,x_d),y=(y_1,\dots ,y_d)  \in \ZZ^d$ such that $\|x-y\|_1 := \sum_{i=1}^d |x_i - y_i| = 1$. We consider a family of i.i.d. random variables $(t(e), e \in \EE^d)$ associated to the edges of the graph, taking values in $[0,+\infty]$ (we emphasize that $+ \infty$ is included here). We denote by $F$ the common distribution of these variables. We interpret $t(e)$ as the time needed to cross the edge $e$. If $x,y$ are vertices in $\ZZ^d$, a path $r$ from $x$ to $y$ is a sequence $r=(v_0,e_1,v_1,\dots, e_n,v_n)$ of vertices $(v_i, i=0,\dots ,n)$ and edges $(e_i, i=1,\dots ,n)$ for some $n\in \NN$ such that $v_0 = x$, $v_n=y$ and for all $i\in \{1,\dots , n\}$, $e_i = \langle v_{i-1}, v_i \rangle$. For any path $r$, we define $T(r) = \sum_{e\in r} t(e)$. We obtain a random pseudo-metric $T$ on $\ZZ^d$ in the following way :
$$ \forall x,y \in \ZZ^d \,, \quad T(x,y) \,=\, \inf \{ T(r) \,|\, r \textrm{ is a path from }x \textrm{ to } y\} \,. $$
The variable $T(x,y)$ is the minimum time needed to go from $x$ to $y$. Because the passage time $t(e)$ of an edge $e$ can be infinite, so does the time $T(x,y)$ for $x,y \in \ZZ^d$. From now on, we suppose that the edges with a finite passage time percolate, i.e., we suppose that $F([0,+\infty [) > p_c(d)$, where $p_c(d)$ is the critical parameter for independent Bernoulli bound percolation on $(\ZZ^d, \EE^d)$. An equivalent way to define our model is to perform first an independent Bernoulli percolation on the edges of $\EE^d$ of parameter  $p>p_c(d)$, then associate to each removed edge an infinite passage time, and associate independently to each remaining edge $e$ a finite passage time $t(e)$ according to a fixed law on $[0, + \infty [$. A central object in our study is the set of points reached from the origin $0$ of the graph within a time $t \in \RR^+$ :
$$ B_t^v \,=\, \{ z\in \ZZ^d \,|\, T(0,z)\leq t \} \,.$$
The exponent $v$ indicates that $B_t^v$ is a set of vertices. It may be useful to consider a fattened set $B_t$ by adding a small unit cube around each point of $B_t^v$, so we also define the following random set :
$$ \forall t\in \RR^+\,, \quad B_t \,=\, \{ z+u \,|\, z\in \ZZ^d \textrm{ s.t. } T(0,z) \leq t \,,\,\, u \in [-1/2,1/2]^d\} \,.$$
We now define a new variable $\wT (x,y)$ which is more regular than $T(x,y)$. Since $F([0,+\infty [)> p_c(d)$, there exists $M \in \RR$ such that $F([0,M] )> p_c(d)$. Fix such a $M$. Let $\C_M$ be the infinite cluster for the Bernoulli percolation $(\ind{ \{ t(e) \leq M \} }, e\in \EE^d)$, which exists and is unique a.s.. To any $x\in \ZZ^d$, we associate a random point $\wx\in \ZZ^d$ such that $\wx \in \C_M$ and $\| x - \wx \|_1$ is minimal ; if there are more than one point in $\C_M$ at minimal distance to $x$ we choose $\wx$ among them according to a deterministic rule. We define the regularized times $\wT$ by
\begin{equation}
\label{defwT}
 \forall x,y \in \ZZ^d \,, \quad \wT (x,y) \,=\, T (\wx, \wy)  \,.
 \end{equation}
We emphasize the fact that $\wx, \wy$ and $\wT(x,y)$ depends on the real $M$ chosen previously. As previously, we define the set $\wB_t$ by
$$ \forall t\in \RR^+\,, \quad \wB_t \,=\, \{ z+u \,|\, z\in \ZZ^d \textrm{ s.t. } \wT(0,z) \leq t \,,\,\, u \in [-1/2,1/2]^d\} \,. $$
We can now state our main results. We start with the study of the times $\wT$ and the sets $\wB_t$.
\begin{thm}[Definition of the time constant]
\label{thmwT}
Suppose that $F([0,+\infty [) > p_c(d)$. Then there exists a deterministic function $\wmu : \RR^d \rightarrow [0,+\infty[$ such that
$$  \forall x \in \ZZ^d \, , \quad \lim_{n\rightarrow \infty} \frac{\wT (0,nx)}{n} \,=\, \wmu (x) \quad \textrm{a.s. and in } L^1 \,. $$
Moreover, $\wmu$ is homogeneous, i.e. $\wmu( \lambda x) = \lambda \wmu(x)$ for all $x\in \RR^d$ and $\lambda \in \RR^+$, $\wmu$ is continuous, and either $\wmu$ is identically equal to $0$ or $\wmu(x)>0$ for all $x\neq 0$.
\end{thm}
Of course $\wmu$ depends on the dimension $d$ and on the law $F$ of the passage times associated to the edges. A priori $\wmu$ could also depend on the real $M$ such that $F([0,M]) > p_c(d)$ chosen to define $\wT$, but we will see in Theorem \ref{thmT} that it is not the case. If $e_1$ denotes the vertex of coordinates $(1,0,\dots,0)$, the constant $\wmu (e_1)$ is called the {\em time constant}. 
\begin{rem}
We also obtain in Proposition \ref{propbn} the a.s. convergence of the so called "point-to-line" regularized passage times (see section \ref{secpos}).
\end{rem}
The next result investigates when the time constant is positive.
\begin{thm}[Positivity of the time constant]
\label{thmpositivity}
Suppose that $F([0,+\infty [)> p_c(d)$. Then
$$ \wmu = 0 \iff F(\{0\})\geq p_c(d) \,. $$
\end{thm}
When $F([0,+\infty [) > p_c(d)$ and $F(\{0\})< p_c(d)$, $\wmu$ is a norm on $\RR^d$, and we denote by $\B_{\tilde \mu}$ the unit ball for this norm
$$\B_{\tilde \mu} \,=\, \{ x\in \RR^d \,|\, \wmu(x) \leq 1 \} \,.$$
We now state the so called {\em shape theorem} for $\wB_t$.
\begin{thm}[Strong shape theorem for $\wB_t$]
\label{thmwB}
Suppose that $F([0,+\infty [) > p_c(d)$.
\begin{itemize}
\item[(i)]  We have
$$ \lim_{n\rightarrow \infty} \sup_{x\in \ZZ^d, \|x\|_1 \geq n} \left| \frac{\wT(0,x) - \wmu(x)}{\|x\|_1} \right| \,=\, 0 \quad \textrm{a.s.} $$
\item[(ii)] If moreover $F(\{0\})< p_c(d)$, then
$$  \forall \eps >0 \,, \textrm{ a.s.}\,, \exists t_0 \in \RR^+ \textrm{ s.t. } \forall t\geq t_0 \,, \quad  (1-\eps) \,\B_{\tilde \mu} \,\subset \frac{\wB_t}{t} \,\subset \, (1+\eps ) \,\B_{\tilde \mu} \,.$$
\end{itemize}
\end{thm}
Concerning the times $T$ and the sets $B_t$, we obtain results that are analog to Theorems \ref{thmwT} and \ref{thmwB}. However, since the times $T$ are less regular than the times $\wT$, the convergences hold in a weaker form. Some times $T^*$ will be natural intermediates between $\wT$ and $T$. Let $\C_\infty$ be the infinite cluster for the Bernoulli percolation $(\ind{\{ t(e) <\infty \} }, e\in \EE^d)$, which exists and is unique a.s.. For all $x\in \ZZ^d$, let $x^* \in \ZZ^d$ be the random point of $\C_\infty$ such that $\|x-x^*\|_1$ is minimal, with a deterministic rule to break ties. We define $T^*(x,y) = T(x^*,y^*)$ for all $x,y\in \ZZ^d$, and the corresponding set
$$ B_t^* \,=\, \{ z+u \,|\, z \in \ZZ^d \textrm{ s.t. } T^* (0,z) \leq t \,,\,\, u \in [-1/2, 1/2 ]^d  \}$$
for all $t\in \RR^+$. The times $T^*$ are less regular than $\wT$ but more regular than $T$, thus their study is a natural step in the achievement of the study of the times $T$. Let $\theta$ be the density of $\C_\infty$ : $\theta = \PP[0 \in \C_{\infty}]$.
\begin{thm}[Weak convergence to the time constant]
\label{thmT}
Suppose that $F([0,+\infty [) > p_c(d)$. Then
$$  \forall x \in \ZZ^d \,,\quad \lim_{n\rightarrow \infty} \frac{T^*(0,nx)}{n} \,=\, \wmu (x)  \quad \textrm{in probability}\,, $$
and
$$ \forall x \in \ZZ^d \,,\quad \lim_{n\rightarrow \infty} \frac{T(0,nx)}{n} \,=\, \theta^2 \delta_{\tilde \mu(x)} + (1-\theta^2) \delta_{+\infty} \quad \textrm{in law} \,. $$
As a consequence, the function $\wmu$ does not depend on the constant $M$ satisfying $F([0,M]) >p_c(d)$ that was chosen in the definition of $\wT$.
\end{thm}
\begin{rem}
We denote our limit by $\wmu$. The existence of a limit for the rescaled times $T(0,nx)/n$ is already known in many cases, see section \ref{secart} for a presentation of those results. By Theorem \ref{thmT} we know that the limit $\wmu$ we obtain is the same as the limit obtained in other settings with more restrictive assumptions on $F$.This limit is usually denoted by $\mu$, but we decided to keep the notation $\wmu$ to emphasize the fact that $\wmu$ is obtained as the limit of the times $\wT$.
\end{rem}
We denote by $A\triangle B$ is the symmetric difference between two sets $A$ and $B$. We denote the Lebesgue measure on $\RR^d$ by $\L^d$. We state the shape theorem in the framework of weak convergence of measures. We say that a sequence $(\mu_n, n\in \NN)$ of measures on $\RR^d$ converges weakly towards a measure $\mu$ if and only if for any continuous bounded function $f : \RR^d \mapsto \RR$ we have 
$$ \lim_{n\rightarrow \infty} \int_{\RR^d} f \, d\mu_n \,=\, \int_{\RR^d} f \,d\mu \,. $$
We denote this convergence by $\mu_n \underset{n \rightarrow \infty}{ \rightharpoonup}  \mu$.
\begin{thm}[Weak shape theorem for $B_t^*$ and $B_t$]
\label{thmB}
Suppose that $F([0,+\infty [) > p_c(d)$ and $F(\{0\})< p_c(d)$.
\begin{itemize}
\item[(i)] We have
$$ \lim_{t\rightarrow \infty } \L^d \left( \frac{B_t ^*}{t} \triangle \B_{\tilde \mu}  \right)  \,=\, 0 \quad \textrm{a.s.} $$
\item[(ii)]
On the event $\{ 0 \in \C_\infty \}$ we have a.s. the following weak convergence :
$$\frac{1}{t^d} \sum_{x\in B_t ^v } \delta_{x/t}   \, \underset{t \rightarrow \infty}{ \rightharpoonup} \, \theta \,  \ind{ \B_{\tilde \mu }} \, \L^d \,. $$
\end{itemize}
\end{thm}

\begin{rem}
We would like to warn the reader that the proofs of these theorems are intertwined. Indeed, we prove the convergence of $\wB_t$ towards $\B_{\tilde \mu}$ under the condition that $\wmu >0$. Then we prove that $\wmu >0$ when $F(\{0\}) < p_d(c)$, and we use these two results to prove the large deviation estimate (Proposition \ref{propGD}) that is the key argument to prove that $\wmu >0$ implies $F(\{ 0 \})<p_c(d)$, which is the delicate step in the proof of Theorem \ref{thmpositivity}. Finally we use Theorem \ref{thmpositivity} in the proof of the weak shape theorem for $B_t^*$ and $B_t$ since we need some compactness argument.
\end{rem}


\subsection{State of the art in first passage percolation}
\label{secart}

The model of first passage percolation has been studied a lot since Hammersley and Welsh \cite{HammersleyWelsh} introduced it in 1965. The results presented in this article are deeply linked to previous works that we try to describe briefly in this section.

First let us consider the case of a law $F$ on $[0,+\infty[$. Thanks to a subadditive argument, Hammersley and Welsh proved in \cite{HammersleyWelsh} that if $d=2$ and $F$ has finite mean, then $\lim_{n\rightarrow \infty} T(0,ne_1) / n $ exists a.s. and in $L^1$, the limit is a constant denoted by $\mu(e_1) $ and called the time constant. The moment condition was improved some years later by several people independently, and the study was extended to any dimension $d\geq 2$ (see for example Kesten's St Flour notes \cite{Kesten:StFlour}). The convergence to the time constant can be stated as follows : if $\EE[\min (t_1,\dots , t_{2d})] < \infty$ where $(t_i)$ are i.i.d. of law $F$, there exists a constant $\mu(e_1) \in \RR$ such that
$$  \lim_{n\rightarrow \infty} \frac{T(0,n e_1)}{ n} \,=\, \mu(e_1) \quad \textrm{a.s. and in } L^1\,.$$
Moreover, the condition $\EE[\min (t_1,\dots , t_{2d})] < \infty$ is necessary for this convergence to hold a.s. or in $L^1$. This convergence can be generalized by the same arguments, and under the same hypothesis, to rational directions : there exists an homogeneous function $\mu : \QQ^d \rightarrow \RR$ such that for all $x \in \ZZ^d$, we have $\lim_{n\rightarrow \infty} T(0, nx) / n = \mu (x)$ a.s. and in $L^1$. The function $\mu$ can be extended to $\RR^d$ by continuity (see \cite{Kesten:StFlour}). A simple convexity argument proves that either $\mu (x) = 0$ for all $x\in \RR^d$, or $\mu(x) >0$ for all $x\neq 0$. Kesten proved in \cite{Kesten:StFlour}, Theorem 1.15, that $\mu > 0$ if and only if $F(\{ 0 \}) < p_c(d)$, i.e., the percolation $(\ind{\{ t(e) = 0 \}}, e\in \EE^d)$ is sub-critical. If $F(\{ 0 \}) < p_c(d)$, $\mu $ is a norm on $\RR^d$, and the unit ball for this norm 
$$\B_\mu \,=\, \{ x\in \RR^d \,|\, \mu (x) \leq 1  \}$$
is compact. A natural question at this stage is whether the convergence $\lim_{n\rightarrow \infty} T(0, nx) / n = \mu (x)$ is uniform in all directions. The shape theorem, inspired by Richardson's work \cite{Richardson}, answers positively this question under a stronger moment condition. It can be stated as follows (see \cite{Kesten:StFlour}, Theorem 1.7) : if $\EE [\min (t_1^d, \dots , t_{2d}^d)] < \infty$, and if $F(\{ 0 \}) < p_c(d)$, then for all $\eps >0$, a.s., there exists $t_0 \in \RR^+$ such that
\begin{equation}
\label{eqforme1}
 \forall t\geq t_0 \, , \quad (1-\eps ) \B_{\mu} \,\subset \, \frac{B_t}{t} \, \subset \, (1+ \eps ) \B_{\mu} \,.
 \end{equation}
Moreover, the condition $\EE [\min (t_1^d, \dots , t_{2d}^d)] < \infty$ is necessary for this convergence to hold a.s. An equivalent shape theorem can be stated when $F(\{0\}) \geq p_c(d)$, but the "shape" appearing in this case is $\B_{\mu} = \RR^d$ itself.

A first direction in which these results can be extended is by considering a law $F$ on $[0,+\infty[$ which does not satisfy any moment condition, at the price of obtaining weaker convergences. Cox and Durrett in dimension $d=2$, and then Kesten in any dimension $d\geq 2$ performed this work successfully in \cite{CoxDurrett} and \cite{Kesten:StFlour} respectively. More precisely, they proved that there always exists a function $\hat \mu : \RR^d \rightarrow \RR^+$ such that for all $x \in \ZZ^d$, we have $\lim_{n\rightarrow \infty} T(0, nx) / n =\hat \mu (x)$ in probability. If $\EE[\min (t_1,\dots , t_{2d})] < \infty$ then $\hat \mu = \mu$. The function $\hat \mu$ is built as the a.s. limit of a more regular sequence of times $\hat T (0,nx) /n$ that we now describe roughly. They consider a $M\in \RR^+$ large enough so that $ F( [0,M])$ is very close to $1$. Thus the percolation $(\ind{\{ t(e) \leq M \}}, e \in \EE^d)$ is highly super-critical, so if we denote by $\C_M$ its infinite cluster, each point $x\in \ZZ^d$ is a.s. surrounded by a small contour $S(x) \subset \C_M$. They define $\hat T (x,y) = T(S(x), S(y))$ for $x,y \in \ZZ^d$. The times $\hat T (0,x)$ have good moment properties, thus $\hat \mu(x)$ can be defined as the a.s. and $L^1$ limit of $\hat T (0,nx) /n$ for all $x\in \ZZ^d$ by a classical subadditive argument; then $\hat \mu$ can be extended to $\QQ^d$ by homogeneity, and finally to $\RR^d$ by continuity. The convergence of $T(0, nx)/n$ towards $\hat \mu (x)$ in probability is a consequence of the fact that $T$ and $\hat T$ are close enough. Kesten's result on the positivity of the time constant remains valid for $\hat \mu$. Moreover, Cox and Durrett \cite{CoxDurrett} and Kesten \cite{Kesten:StFlour} proved an a.s. shape theorem for 
$$\hat B_t \,=\, \{ z+u \,|\, z\in \ZZ^d \textrm{ s.t. } \hat T(0,z) \leq t \,,\,\, u \in [-1/2,1/2]^d\}$$
with shape limit $\B_{\hat \mu} = \{ x\in \RR^d \,|\, \hat \mu (x) \leq 1 \}$ when $\hat \mu $ is a norm, i.e. $F(\{0\}) < p_c(d)$ (and an equivalent shape result with shape limit equal to $\B_{\hat \mu} = \RR^d$ when $F(\{0\}) \geq p_c(d)$). In dimension $d=2$, Cox and Durrett also deduced a weak shape theorem for $B_t$ (see Theorem 4 in \cite{CoxDurrett}) :
\begin{equation}
\label{eqCoxDurrett1}
 \forall K \in \RR^+ \,, \quad \lim_{t\rightarrow \infty} \L^2 \left( \left(  \frac{B_t}{t} \triangle \B_{\hat \mu} \right) \cap \{ x \in \RR^d \,|\, \| x\|_1 \leq K \} \right)  \,=\, 0 \quad \textrm{a.s.}\,,
 \end{equation}
where $\triangle$ denotes the symmetric difference between two sets, and
\begin{equation}
\label{eqCoxDurrett2}
\forall \eps >0 \,, \textrm{ a.s.} \,, \,\,\exists t_0 \in \RR^+ \textrm{ s.t. } \forall t\geq t_0 \, \quad \frac{B_t}{t} \, \subset \, \B_{\mu} \,.
\end{equation}
In fact, in the case $F(\{0\}) < p_c(d)$, the intersection with $\{ x \in \RR^d \,|\, \| x\|_1 \leq K \}$ is not needed in (\ref{eqCoxDurrett1}), since $\B_{\hat \mu}$ is compact. The inclusion in $(\ref{eqCoxDurrett2})$ follows directly from the a.s. shape theorem for $\hat B_t$ since $\hat T(0,x) \leq T(0,x)$ for all $x\in \ZZ^d$. Kesten did not write the generalization to any dimension $d\geq 2$ of this weak shape theorem for $B_t$ without moment condition on $F$ but all the required ingredients are present in \cite{Kesten:StFlour}.

A second direction in which these results can be extended is by considering random passage times $(t(e), e\in \EE^d)$ that are not i.i.d. but only stationary and ergodic. Boivin defined in \cite{Boivin90} a time constant in this case and proved a corresponding shape theorem under some moment assumptions on $F$. We do not present these results in details since this generalization is not directly linked with the purpose of the present article.

A third possible way to generalize these results is to consider infinite passage time. This case has been studied by Garet and Marchand in \cite{GaretMarchand04}. They presented it as a model of first passage percolation in random environment: they consider first a super-critical Bernoulli percolation on $ \EE^d$, and then they associate to each remaining edge $e$ a finite passage time $t(e)$ such that the family $(t(e), e \in \EE^d)$ is stationary and ergodic. If $x$ and $y$ are two vertices that do not belong to the same cluster of the Bernoulli percolation, there is no path from $x$ to $y$ and $T(x,y) = +\infty$. To define a time constant $\mu'(x)$ in a rational direction $x$, they first consider the probability $\overline{\PP}$ conditioned by the event $\{ 0 \in \C_\infty\}$, where $\C_{\infty}$ is the infinite cluster of the super-critical percolation mentioned above. In the direction of $x$ they only take into account the points $(x_n, n\in \NN)$ that belong to $\C_\infty$, with $\lim_{n\rightarrow \infty} \|x_n\|_1 = \infty$. Then under a moment condition on the law of the passage times, they prove that $\overline{\PP}$-a.s., the times $T(0,x_n)$ properly rescaled converge to a constant $\mu'(x)$. They also prove a shape theorem for $B_t$ when $\mu'$ is a norm (i.e. when $\mu'(e_1)>0$) :
$$ \lim_{t\rightarrow \infty} \mathcal D _H \left( \frac{B_t}{t} , \B_{\mu'} \right)\,=\, 0  \quad \overline{\PP}\textrm{-a.s.}\,, $$
where $\mathcal D_H$ denotes the Hausdorff distance between two sets, and $\B_{\mu'} = \{ x\in \RR^d \,|\, \mu'(x) \leq 1 \}$. Let us remark that the infinite cluster $\C_{\infty}$ has holes, and so does the set $B_t$, thus a shape theorem as stated in (\ref{eqforme1}) cannot hold. The use of the Hausdorff distance allows to fill the small holes in $B_t$. Garet and Marchand's results are all the more general since they did not consider i.i.d. passage times but the ergodic stationary case as initiated by Boivin. However, their moment condition on the finite passage times is quite restrictive (see hypothesis ($H_\alpha$) on page 4 in \cite{GaretMarchand04}). As they explained just after defining this hypothesis, in the i.i.d. case, ($H_{\alpha}$) corresponds to the existence of a moment of order $2 \alpha$. The existence of $\mu'$ is proved if ($H_\alpha$) holds for some $\alpha >1$, thus in the i.i.d. case with a moment of order $2 + \eps$. The shape theorem is proved if ($H_\alpha$) holds for some $\alpha> d^2 + 2d - 1$, thus in the i.i.d. case with a moment of order $2 (d^2 + 2d - 1)+ \eps$. We emphasize that these hypotheses are of course fulfilled if the finite passage times are bounded, which is the case in particular if the finite passage times are equal to $1$. In this case $T(x,y)$, $x,y\in \ZZ^d$ is equal to the length of the shortest path that links $x$ to $y$ in the percolation model if $x$ and $y$ are connected, and it is equal to $+\infty$ if $x$ and $y$ are not connected. The variable $T(x,y)$ is called the {\em chemical distance} between $x$ and $y$ and is usually denoted by $D(x,y)$. This chemical distance was previously studied, and we will present a powerful result of Antal and Pisztora \cite{AntalPisztora} in the next section (see Theorem \ref{thmAP}). To finish with the presentation of Garet and Marchand's works, we should say that the generality of the stationary ergodic setting they chose makes it also difficult to give a characterization of the positivity of the time constant in terms of the law of the passage times. Garet and Marchand give sufficient or necessary conditions for the positivity of $\mu'$ (see Section 4 in \cite{GaretMarchand04}), which in the i.i.d. case correspond to :
$$ [ F(\{0\}) > p_c(d) \, \Rightarrow \mu'(e_1)=0] \quad \textrm{and} \quad [F(\{0\}) <p_c(d) \, \Rightarrow \mu'(e_1)>0 ]\,,$$ 
but they do not study the critical case $F(\{0\}) = p_c(d)$.

A lot more was proved concerning first passage percolation. Many people investigated large deviations, moderate deviations and variance of the times $T(0,x)$ (see for instance the works of Kesten \cite{Kesten:StFlour}, Benjamini, Kalai and Schramm \cite{Benjamini:variance}, Chatterjee and Dey \cite{ChatterjeeDey}, Garet and Marchand \cite{GaretMarchand07,GaretMarchand10} and the recent paper of Ahlberg \cite{Ahlberg13}). We do not go further in this direction, even if in section \ref{secpos2} we prove a large deviation estimate, since we see it as a tool rather than a goal.

The aim of the present paper is to fulfill the gap between the works of Cox, Durrett and Kesten on one hand, and Garet and Marchand on the other hand. More precisely, we prove a weak convergence to a time constant (Theorem \ref{thmT}) and a weak shape theorem (Theorem \ref{thmB}) when the passage times are i.i.d., maybe infinite, and without any moment assumption on the finite passage times. We also obtain necessary and sufficient condition for the positivity of the time constant in this setting (Theorem \ref{thmpositivity}). Our strategy follows the approach of Cox and Durrett in \cite{CoxDurrett}, that was adapted by Kesten \cite{Kesten:StFlour} in dimension $d\geq2$ : we define an auxiliary time $\wT$ (see equation (\ref{defwT})), we prove that it has good moment properties, so we can prove the a.s. convergence of $\wT(0,nx)/n$ towards a time constant $\wmu (x)$ in rational directions, and a strong shape theorem for $\wB_t$. Then we compare $T$ to $\wT$ to deduce weak analogs for the times $T$. However, we cannot use the same regularized times $\hat T$ as Cox, Durrett and Kesten did. Indeed, they use the fact that for $M$ large enough, when the passage times are finite, $F([0,  M]) $ can be chosen arbitrarily close to $1$, so the percolation $(\ind{\{ t(e) \leq M \}}, e \in \EE^d)$ can be chosen as super-critical as needed. What they need precisely is the existence of contours included in $\C_{M}$ (the infinite cluster of edges of passage time smaller than $M$) around each point. In dimension $2$, it is in fact enough to require that $F([0, +\infty[) > p_c(2) = 1/2$, thus $F([0,M]) > 1/2$ for $M$ large enough. Indeed when $d=2$, because $p_c(2) = 1/2$, if open edges are in a super-critical regime then closed edges are in a sub-critical regime and they do not percolate. However, for $d\geq 3$, this is not true anymore: it may happen that open edges and closed edges percolate at the same time (remember that $p_c(d) < 1/2$ for $d\geq 3$), and in this case we do not have the existence of open contours around each point. In our setting, we have $F([0,  M])\leq F([0, +\infty[) $, and $F([0, +\infty[) $ is fixed so it cannot be pushed towards $1$, we only know that it is strictly bigger than $p_c(d)$. This is not enough to ensure the existence of contours included in $\C_M$ or $\C_{\infty}$ (the infinite cluster of edges of finite passage time) around each point for $d\geq 3$. In the case where $F(\{+\infty\})$ is very small, it is likely that the proofs of Cox, Durrett and Kesten could work with minor adaptations. But with general laws such that $F([0, +\infty[)> p_c(d)$, we need to define new regularized times $\wT$. Our definition of $\wT(x,y)$ as $T(\wx, \wy)$ where $\wx$ is the point of $\C_M$ (the infinite cluster of edges of passage time $t(e) \leq M$) is inspired by Garet and Marchand \cite{GaretMarchand10} who associated to the chemical distance $D(x,y)$ between $x,y\in \ZZ^d$, that may be infinite, the finite distance $D^* (x,y) = D(x^*,y^*)$, where $x^*$ is the point of the infinite cluster of the underlying Bernoulli percolation which is the closest to $x$ for the $\| \cdot \|_1$ distance (with a deterministic rule to break ties). The corresponding $x^*$ in our setting is the point of $\C_{\infty}$ which is closest to $x$. This choice for $x^*$ seems more natural than $\wx$, since the definition of $\wx$ depends on a real number $M$ satisfying $F([0,M])>p_c(d)$. However, the times $T^* (0,x)=T(0^*, x^*)$ are not regular enough for our purpose. Thus we follow Cox, Durrett and Kesten by introducing an arbitrary M and working with $\C_M$ to define the times $\wT$. Using the results obtained for $\wT$, we can study the times $T^*$ and then $T$. Another originality of our work is the use of measures to state the weak shape theorem (Theorem \ref{thmB} (ii)). With this formulation, both the limit shape $\B_{\tilde \mu}$ and the density $\theta$ of the infinite cluster $\C_\infty$ appear naturally, so we believe it is particularly well adapted to this question.

We also have to mention the work of Mourrat \cite{Mourrat12} in our state of the art. Mourrat deals with random walk in random potential. Consider a family of i.i.d. variables $(V(x), x\in \ZZ^d)$ associated to the vertices of $\ZZ^d$, such that $V(x)$ takes values in $[0,+\infty]$. Let $S=(S_n, n \in \NN)$ be a simple random walk on $\ZZ^d$, starting at $x\in \ZZ^d$ under the probability $\mathbf P_x$ (of expectation $\mathbf E_x$). Given $y\in \ZZ^d$, let $H_y= \inf \{ n\geq 0 \,|\, S_n = y \}$. Then define
$$\forall x,y\in \ZZ^d\,,\quad  a(x,y) \,=\, - \log \mathbf E_x \left[ \exp \left( - \sum_{n=0}^{H_y - 1} V(S_n)\right) \ind{\{ H_y < \infty \}} \right] \,,$$
with $a(x,x) = 0$. Mourrat proves for all $x\in \ZZ^d$ the convergence of $a(0,nx)/n$ towards a deterministic constant $\alpha (x) $ called the Lyapunov exponent (see Theorem 1.1 in \cite{Mourrat12}). Here $\alpha (e_1) >0$ whatever the law of the potentials $V$, thus $\alpha$ defines a norm on $\RR^d$. If 
$$A_t \,=\, \{ x+u \,|\, x \in \ZZ^d\,,\,\, a(0,x) \leq t \,,\,\, u \in [-1/2,1/2]^d  \}$$
and $\B_{\alpha} = \{ x\in \RR^d \,|\, \alpha (x) \leq 1\}$, then Mourrat proves a shape theorem : $A_t/t$ converges towards $\B_\alpha$ (see Theorem 1.2 in \cite{Mourrat12}). The sense in which these convergences happen in both results depend on the moments of $V(x)$. Mourrat obtains necessary and sufficient conditions on the law of $V$ to obtain strong convergences which are similar to the one obtained in classical first passage percolation : $\EE [\min (V_1, \dots , V_{2d})] <\infty$ for the a.s. convergence towards $\alpha$, and $\EE [\min (V_1, \dots , V_{2d})^d ] <\infty$ for the strong shape theorem, where the $V_i$ are i.i.d. with the same law as the $V(x)$. He also obtains a weak shape theorem for $A_t$ without any moment assumption, including infinite potentials as soon as $\PP[V(x) < \infty] > p'_c(d)$ (here $p'_c(d)$ is the critical parameter for Bernoulli percolation on the vertices of $\ZZ^d$) : on the event $\{0 \in \C_\infty \}$, where $\C_\infty$ is the infinite cluster of the percolation of vertices of finite potential, and for any sequence $\eps_t$ such that $\lim_{t \rightarrow \infty } \eps_t = +\infty$ and $\lim_{t \rightarrow \infty } \eps_t/t =0 $, he proves that
\begin{equation}
\label{eqMourrat}
\lim_{t\rightarrow \infty} \L^d \left( \frac{A_t^{\eps_t}}{t} \triangle \B_{\alpha} \right) \,=\, 0 \quad \textrm{a.s.} 
\end{equation}
where $A^\eps = \{ y \in \RR^d \,|\, \exists x \in A \textrm{ s.t. } \|x-y\|_2 \leq \eps \}$ is the $\eps$-neighborhood of $A$ of size $\eps$ for the euclidean distance. Notice that taking these neighborhoods allows Mourrat to take care of the holes in the infinite cluster $\C_{\infty}$ as Garet and Marchand did by considering the Hausdorff distance. The work of Mourrat is deeply linked with our setting: consider this model for potentials $(\beta V(x), x\in \ZZ^d)$ and let $\beta>0$ goes to $\infty$, the measure on the paths $(S_n, n\in \{0,\dots , H_y\})$ between $x$ and $y$ charges only paths of minimal weights $\sum_{n=0}^{H_y - 1} V(S_n)$, i.e., the geodesics for the first passage percolation of passage times $(V(x), x\in \ZZ^d)$ on the vertices of $\ZZ^d$. This limit can be seen as the limit at null temperature, since $\beta$ usually represents the inverse of the temperature in this kind of model. The first difference between Mourrat's setting and our setting is that he considers random potential on the vertices of $\ZZ^d$ whereas we consider random passage times on the edges of $\ZZ^d$. We believe that this difference in the model should not change the validity of either arguments. The second difference is that Mourrat considers the model of random walk in random potential at positive temperature, whereas we consider its limit at null temperature, i.e. the first passage percolation model. Notice that even if the model of random walk in random potentials is more general than the first passage percolation model, it is not straightforward to obtain the shape theorem for first passage percolation as a corollary of Mourrat's shape theorem. But it is likely that Mourrat's arguments could work similarly in the case of null temperature. However, Mourrat's approach is somehow more intricate than ours, so we believe that our method is interesting in itself. Indeed, Mourrat met the same problem as we did in the use of Cox and Durrett's argument : in his setting, the percolation $(\ind{\{ V(x) \leq M \}}, x\in \ZZ^d)$ cannot be chosen as super-critical as he wants, so he cannot define shells in the infinite cluster $\C_M$ of this percolation that surround points. To circumvent this difficulty, he uses a renormalization scheme : he defines blocks at a mesoscopic scale, and declares a block to be "good" if $M$ behaves "well" inside this block (see section 5 of \cite{Mourrat12}). He chooses the size of the blocks such that the percolation of good blocks is highly super-critical. Then he considers around each point of $\ZZ^d$ a shape of good blocks that surrounds this point, and that belongs to the infinite cluster of good blocks. He finally defines his regularized times $\hat a (x,y)$ as (what would be in our setting) the minimal time needed to go from the part of $\C_M$ which belongs to the shell of good blocks around $x$, to the part of $\C_M$ which belongs to the shell of good blocks around $y$. Notice that this renormalization construction does not avoid the use of existing results, the same as the one we will use and that we present in the next section, some of them being themselves proved by a renormalization scheme (see Theorem \ref{thmAP} for instance). Thus our approach, with a simpler definition of the regularized times $\wT$ and thus of the time constant, seems more natural. Moreover, we improve our comprehension of the model by the study of the times $T^*$ and the sets $B_t^*$, that do not appear in Mourrat's work. We think that our statement of the weak shape theorem for $B_t$ in terms of measures is also somehow more descriptive than Mourrat's weak shape theorem, since he considers a fattened version $A_t^{\eps_t}$ of $A_t$ to fill the holes in the infinite cluster $\C_\infty$, whereas we choose to make them appear in our limit through their density. Finally, Mourrat does not investigate the positivity of the constant in first-passage percolation, since in his model at positive temperature the constants $\alpha (x)$ are always positive - this property can be lost when taking the limit at null temperature, and it is indeed the case if $F(\{  0 \} ) \geq p_c(d)$ as proved in Theorem \ref{thmpositivity}.

Other related works are those of Cox and Durrett \cite{CoxDurrett:epidemics}, Zhang \cite{Zhang:epidemics} and Andjel, Chabot and Saada \cite{AndjelChabotSaada} about the spread of forest fires and epidemics. We do not present these models in details but, roughly speaking, they correspond to the study of first passage percolation with the following properties :
\begin{itemize}
\item[$(i)$] the percolation is oriented, i.e., the passage time from a vertex $x$ to a vertex $y$ is not equal to the passage time from $y$ to $x$;
\item[$(ii)$] the percolation is locally dependent;
\item[$(iii)$] the passage time of an edge is either infinite either equal to $1$ in the spread of forest fires, or smaller than an exponential variable in the spread of epidemics.
\end{itemize}
The models with properties (i) and (ii) are more difficult to study than the case of i.i.d. first passage percolation we consider here, indeed it is necessary to develop adequate percolation estimates in this setting. Property (iii) makes the study easier, since finite edges have good moment properties. In \cite{CoxDurrett:epidemics}, Cox and Durrett study these models in dimension $2$, and they prove the corresponding shape theorems. Since they work in dimension $2$, they can use the methods they developed in \cite{CoxDurrett} with the definition of contours around points. Zhang studies also the two-dimensional case in \cite{Zhang:epidemics} but he adds some finite range interactions, i.e., he adds edges in the graph between vertices that are not nearest neighbors. Andjel, Chabot and Saada study in \cite{AndjelChabotSaada} the epidemic model in dimension $d\geq 3$. They face the same problem as we and Mourrat do : the definition of contours is not adapted anymore. To circumvent this problem, they replace the contours by more complicated random neighborhoods around each vertex. The definition we propose in our setting is a lot more straightforward.  

A natural question raised by this work is the study of the continuity of the time constant. Cox \cite{Cox} and Cox and Kesten \cite{CoxKesten} proved the continuity of the time constant with regard to a variation of the law $F$ of the passage time on $[0,+\infty[$ in dimension $2$, and this result was extended to any dimension $d\geq 2$ by Kesten \cite{Kesten:StFlour}. It may be possible to combine their techniques with the definition of the time constant for possibly infinite passage times we propose in this article. We emphasize the fact that the existence of the time constant without any moment condition and the study of the positivity of the time constant are two of the key arguments used by Cox and Kesten in \cite{CoxKesten} to improve the work of Cox \cite{Cox}.


\subsection{Useful results}
\label{secbackground}

In the previous section, we presented several papers whose results and techniques influenced our work. In this section, we present a few results concerning percolation and chemical distance that we will use. In what follows, we consider an i.i.d. Bernoulli percolation on the edges of $\ZZ^d$ of parameter $p>p_c(d)$.

First we present the result of Antal and Pisztora \cite{AntalPisztora}. Denote by $D(x,y)$ the chemical distance between $x$ and $y\in \ZZ^d$. We recall that if $| \gamma |$ denotes the length of a path $\gamma$ (i.e. the number of edges in $\gamma$), then
$$ \forall x,y \in \ZZ^d \,, \quad D(x,y) \,=\, \inf \{ |\gamma | \,:\, \gamma \textrm{ is a path from }x \textrm{ to }y \} \,. $$
If $x$ and $y$ are not connected, there exists no such path and by definition $D(x,y) = +\infty$. The event that $x$ and $y$ are connected will be denoted by $\{ x \leftrightarrow y \}$. We write Antal and Pisztora's result in an appropriate form for our use.
\begin{thm}[Chemical distance]
\label{thmAP}
Suppose that $p>p_c(d)$. There exist positive constants $A_1$, $A_2$ and $A_3$ such that
$$ \forall y \in \ZZ^d\,,\,\, \forall l\geq A_3 \|y\|_1 \,, \quad \PP[ 0 \leftrightarrow y\,,\, D(0,y) \geq l ] \,\leq \, A_1 e^{-A_2 l} \,. $$
\end{thm}
In fact Antal and Pisztora did not write their result exactly this way : their Theorem 1.1 is uniform in all directions, but in their Theorem 1.2, on which we rely, they do not state explicitly that the result is uniform with regard to the directions. However a close inspection of their arguments shows that it is indeed the case.

Let us denote by 
$\C$ the (a.s. unique) infinite cluster of the percolation we consider. Let 
$$\B_1(x,r) \,=\, \{ y\in \RR^d \,|\, \|x-y\|_1 \leq r \}$$
be the ball of center x and radius $r>0$ for the norm $\| \cdot \|_1$, and $\partial \B_1 (x,r)$ be its boundary. We say that a vertex $x$ is connected to a subset $A$ of $\RR^d$ if there exists $y\in A \cap \ZZ^d$ such that $x$ is connected to $y$, and we denote this event by $\{ x \leftrightarrow A \}$. The following results control the size of the finite clusters and of the holes in the infinite cluster, respectively.
\begin{thm}[Finite clusters]
\label{thmfinite}
Suppose that $p>p_c(d)$. There exist positive constants $A_4$, $A_5$ such that
$$  \forall r>0 \,,\quad \PP [ 0 \notin \C \,,\, 0 \leftrightarrow \partial \B_1(0,r)  ] \,\leq \, A_4 e^{-A_5 r} \,.  $$
\end{thm}
For a reference for Theorem \ref{thmfinite}, see for instance Grimmett \cite{grimmett99}, Theorems 8.18 and 8.19, or Chayes, Chayes, Grimmett, Kesten and Schonmann \cite{CCGKS}.
\begin{thm}[Holes]
\label{thmholes}
Suppose that $p>p_c(d)$. There exist positive constants $A_6$, $A_7$ such that
$$  \forall r>0 \,,\quad \PP [\C \cap \B_1 (0,r) = \emptyset ] \,\leq\, A_6 e^{-A_7 r} \,.$$
\end{thm}
Theorem \ref{thmholes} was proved in dimension $2$ by Durrett and Schonmann \cite{DurrettSchonmann}, and their proof can be extended to dimension $d\geq 2$. For $d\geq 3$, Theorem \ref{thmholes} can also be deduced from Grimmett and Marstrand's slab's result proved in \cite{GrimmettMarstrand}.


\subsection{Organization of the proofs}

The paper is organized as follows. In section \ref{secmu} we investigate the convergence of the rescaled times towards a time constant. In section \ref{secwT}, we prove that the times $\wT$ have good moment properties, then by a classical subadditive argument we prove Theorem \ref{thmwT}. In section \ref{secT}, we compare the times $T$, $T^*$ and $\wT$, and we deduce Theorem \ref{thmT} from Theorem \ref{thmwT}. In section \ref{secwB}, by classical methods, we prove that the convergence of the rescaled times $\wT$ towards $\wmu $ is uniform in all directions (Theorem \ref{thmwB} (i)). Under the assumption that the time constant $\wmu(e_1)$ is strictly positive, we prove that this uniform convergence is equivalent to the strong shape theorem for $\wB_t$ (Theorem \ref{thmwB} (ii)). Then we compare again $T$, $T^*$ and $\wT$ in section \ref{secB} to obtain the weak shape theorems for $B_t^*$ and $B_t$ (Theorem \ref{thmB}). Finally in section \ref{secpos} we study the positivity of the time constant. First we study in Section \ref{secpos1} the positivity of $\wmu(e_1)$ when $F(\{ 0 \}) \neq p_c(d)$. Then we study in section \ref{secLD} the lower large deviations of some passage time below the time constant when the time constant is positive. Finally this result is used in section \ref{secpos2} to prove Theorem \ref{thmpositivity}.


\section{Time constant}
\label{secmu}

\subsection{Definition of the time constant and first properties}
\label{secwT}

We recall that $(t(e), e \in \EE^d)$ is an i.i.d. family of random passage times of law $F$ on $[0,+\infty]$. We assume that $F([0,+\infty]) >p_c(d)$. We select a real number $M$, depending on $F$, such that $F([0,M])>p_c(d)$. For $x\in \RR^d$, $x=(x_1,\dots , x_d)$, let $\|x\|_1 = \sum_{i=1}^d |x_i|$, $\| x \|_2 = \sqrt{\sum_{i=1}^d x_i^2}$ and $\|x\|_{\infty} = \max_{i\in \{1,\dots ,d\}} |x_i| $. We denote by $\B_i (x, r)$ the ball of center $x\in \RR^d$ and radius $r>0$ for the norm $\| \cdot \|_i$:
$$\B_i (x,r) \,=\, \{ y\in \RR^d \,|\, \|x-y\|_i \leq r \} \,.$$
Let $\C_M$ (resp. $\C_\infty$) the (a.s. unique) infinite cluster of the Bernoulli percolation $(\ind{\{ t(e) \leq M \}}, e\in \EE^d)$ (resp. $(\ind{\{ t(e) < \infty \}}, e\in \EE^d)$). Notice that $\C_M \subset \C_\infty$. We denote by $D_M (x,y)$ (resp. $D_\infty (x,y)$) the chemical distance between to vertices $x$ and $y$ in the percolation $ (\ind{\{ t(e) \leq M \}}, e\in \EE^d)$ (resp. $ (\ind{\{ t(e) <\infty \}}, e\in \EE^d)$), and we denote by \smash{$ \{ x \overset{M}{\longleftrightarrow} y \}$} (resp. \smash{$ \{ x \overset{\infty}{\longleftrightarrow} y \}$}) the event that $x$ is connected to $y$ in this percolation. We recall that for all $x,y \in \ZZ^d$, $\wT (x,y) = T(\wx, \wy)$ where $\wx $ is the random point of $\C_M$ that is the closest to $x$ for the norm $\| \cdot \|_1$, with a deterministic rule to break ties. We denote by $|A|$ the cardinality of a set $A$.

Before proving Theorem \ref{thmwT} we need a control on the tail of the distribution of $\wT(0,x)$.
\begin{prop}
\label{prop1}
There exist positive constants $C_1$, $C_2$ and $C_3$ such that 
$$ \forall x \in \ZZ^d \,, \,\, \forall l\geq C_3 \|x\|_1\,, \quad \PP [\wT (0,x) > l ] \,\leq C_1 e^{-C_2 l} \,. $$
\end{prop}

\begin{proof}
For $a,b \in \ZZ^d$, we always have
$$ T(a,b) \,\leq\, M \, D_M (a,b) \,.$$ 
Indeed if $D_M(a,b)=+\infty$ there is nothing to prove. Suppose that $D_M(a,b) <\infty$, thus there exist paths from $a$ to $b$ that are open for the percolation $(\ind{\{ t(e) \leq M \}}, e\in \EE^d)$. Let $\gamma$ be such a path, then $ T(a,b)\leq T(\gamma) \leq M |\gamma | $. Taking the infimum of the length of such paths $\gamma$, we obtain the desired inequality. Consider $x\in \ZZ^d$ and $l\geq C \|x\|_1$, for a constant $C$ to be fixed. We have for all $\eps >0$
\begin{align*}
\PP [\wT (0,x)> l] & \,\leq \, \PP [ \w0 \notin \B_1(0,\eps l/2) ] + \PP [ \wx \notin \B_1(x,\eps l/2) ] \\
& \quad + \sum_{\tiny \begin{array}{c} a\in \B_1 (0,\eps l /2) \cap \ZZ^d\\ b \in \B_1 (x,\eps l /2) \cap \ZZ^d \end{array}} \PP[ \w0 = a\,,\, \wx = b \,,\, D_M (a,b) \geq l/M ]\\
& \,\leq \, \PP[\C_M \cap \B_1(0,\eps l/2) = \emptyset] + \PP[\C_M \cap \B_1(x,\eps l/2) = \emptyset]\\ 
& \quad + \sum_{\tiny \begin{array}{c} a\in \B_1 (0,\eps l /2) \cap \ZZ^d\\ b \in \B_1 (x,\eps l /2) \cap \ZZ^d \end{array}} \PP[ a \overset{M}{\longleftrightarrow}  b \,,\, D_M (a,b) \geq l/M ]\,.
\end{align*}
By Theorem \ref{thmholes} applied to the percolation $(\ind{\{ t(e) \leq M \}}, e\in \EE^d)$ we get that
$$ \PP[\C_M \cap \B_1(0,\eps l/2) = \emptyset]  \,=\,  \PP[\C_M \cap \B_1(x,\eps l/2) = \emptyset] \,\leq\, A_6 e^{-A_7 \eps l/2} \,.$$
For all $a\in \B_1 (0,\eps l /2) \cap \ZZ^d$ and $ b \in \B_1 (x,\eps l /2) \cap \ZZ^d$, we have
$$ \| a-b \|_1  \,\leq \,\| x\|_1 + \eps l \,\leq\, l \left( \frac{1}{C} + \eps \right)\,,$$
thus
$$ l\,\geq\, \frac{1}{C^{-1} + \eps} \| a-b \|_1 \,.  $$
Let $A_3$ the constant defined in Theorem \ref{thmAP} applied to the percolation $(\ind{\{ t(e) \leq M \}}, e\in \EE^d)$. Choose $\eps = C^{-1}$, and $C = 2 M A_3$. Then $l/M \geq A_3  \|a-b\|_1 $, and applying Theorem \ref{thmAP} we get
$$  \PP[ a \overset{M}{\longleftrightarrow}  b \,,\, D_M (a,b) \geq l/M ] \,\leq \, A_1 e^{-A_2 l/M} $$
uniformly in $a\in \B_1 (0,\eps l /2) \cap \ZZ^d$ and $ b \in \B_1 (x,\eps l /2) \cap \ZZ^d$. Since $|\B_1 (0,\eps l /2) \cap \ZZ^d| = |\B_1 (x,\eps l /2) \cap \ZZ^d| \leq C' \eps^d l^d $ for some constant $C'$, Proposition \ref{prop1} is proved.
\end{proof}
\begin{proof}[Proof of Theorem \ref{thmwT}]
The structure of this proof is very classical (see for instance \cite{Kesten:StFlour}, Theorem 1.7), so we try to give all the arguments briefly. Similarly to $T$, $\wT$ satisfies the triangle inequality :
\begin{equation}
\label{inegtrig}
\forall x,y,z\in \ZZ^d \,,\quad \wT(x,y) \,\leq \, \wT(x,z) + \wT(z,y)\,.
\end{equation}
Let $x\in \ZZ^d$. The family $(\wT (mx,nm), 0\leq m <n)$ indexed by $m,n \in \NN$ is a stationary process which is subbadditive by Inequality (\ref{inegtrig}). By Proposition \ref{prop1} we know that this process is integrable. Following Kingman's subadditive ergodic theorem (see \cite{Kingman73} or section 2 in \cite{Kesten:StFlour}) we deduce that 
$$  \lim_{n\rightarrow \infty} \frac{\wT (0,nx)}{n} \textrm{ exists a.s. and in }L^1\,.$$
Moreover by ergodicity this limit, that we denote by $\wmu (x)$, is constant a.s. and we know that
$$ \mu(x)  \,=\, \lim_{n \rightarrow \infty} \frac{\EE [\wT (0,nx)]}{n}  \,=\, \inf_{n>0} \frac{\EE [\wT (0,nx)]}{n}  \,. $$
If $x\in \QQ^d$, let $N\in \NN^*$ be such that $Nx \in \ZZ^d$, and define similarly
$$ \wmu (x) \,=\, \lim_{n\rightarrow \infty} \frac{\wT(0,nNx)}{nN}  \quad \textrm{a.s. and in }L^1.$$
The function $\wmu$ is well defined on $\QQ^d$ ($\wmu (x)$ does not depend of the choice of $N$) and is homogeneous in the sense that for all $x\in \QQ^d$, for all $r\in \QQ^+$,
\begin{equation}
\label{eqhom}
\wmu (rx) \,=\, r \wmu (x) \,.
 \end{equation}
The function $\wmu$ is invariant under any permutation or reflection of the coordinate axis since the model does. By (\ref{inegtrig}) we deduce that for all $x,y\in \QQ^d$
\begin{equation}
\label{inegtrig2}
  \wmu (y) \,\leq\, \wmu(x) + \wmu(y-x)\,,
 \end{equation}
 and
\begin{equation}
\label{eqmaj}
 \wmu (x) \,\leq \wmu (e_1) \|x\|_1\,, 
  \end{equation}
thus
\begin{equation}
\label{eqcont}
|\wmu (x) - \wmu (y) | \,\leq \, \wmu (y-x) \,\leq \, \wmu (e_1) \|y-x\|_1 \,.
 \end{equation}
We can therefore extend $\wmu$ by continuity to $\RR^d$, and (\ref{eqhom}), (\ref{inegtrig2}), (\ref{eqmaj}) and (\ref{eqcont}) remain valid for all $x,y\in \RR^d$ and $r\in \RR^+$. By (\ref{eqmaj}) if $\wmu(e_1)=0$ then $\wmu(x) = 0$ for all $x\in \RR^d$. Conversely, if there exists $x=(x_1,\dots , x_d) \in \RR^d \smallsetminus \{0\}$ such that $\wmu (x) =0$, then by symmetry we can suppose for instance that $x_1 \neq 0$. By (\ref{inegtrig2}) we obtain
$$ 2 |x_1| \wmu (e_1) \,=\, \wmu (2x_1,0, \dots, 0) \,\leq \, \wmu (x_1, x_2,... ,x_d ) + \wmu(x_1, -x_2,... ,-x_d )  \,=\, 0 \,,  $$ 
thus $\wmu(e_1) =0$.
\end{proof}


\subsection{From $\wT$ to $T$}
\label{secT}

We now compare $T^*$ with $\wT$ :
\begin{prop}
\label{propT*wT}
For all $x\in \ZZ^d$, for all $\eps>0$, we have
$$  \lim_{n\rightarrow \infty}\PP \left[ |T^*(0,nx) - \wT(0,nx)| \geq \eps n \right] \,=\, 0 \,. $$
\end{prop}
\begin{proof}
By (\ref{inegtrig2}) we have
\begin{align*}
\PP \left[ |T^*(0,nx) - \wT(0,nx)| \geq \eps n  \right] & \,\leq \,\PP \left[ T(0^*,\w0 ) + T(nx^*,\widetilde{nx} ) \geq \eps n   \right] \\
& \,\leq \, 2\, \PP \left[ T(0^*,\w0) \geq  \eps n/2  \right]
\end{align*}
which goes to $0$ as $n$ goes to infinity since $T(0^*,\w0)<\infty$ because both $0^*$ and $\w0$ belong to $\C_\infty$.
\end{proof}

\begin{proof}[Proof of Theorem \ref{thmT}]
The fact that $\lim_{n\rightarrow \infty}T^*(0,nx)/n = \wmu (x)$ in probability is a simple consequence of Theorem \ref{thmwT} and Proposition \ref{propT*wT}. Since $T^*(0,nx)$ does not depend on the $M$ chosen to define $\wT$, $\wmu(x)$ does not depend on $M$, for all $x\in \ZZ^d$, and thus for all $x\in \RR^d$ by construction. Let $f$ be a continuous bounded real function defined on $\RR \cup \{ + \infty \}$, and let us denote by $\|f\|_\infty$ the supremum of $|f|$. Let $x\in \ZZ^d$. It remains to prove that
$$ \lim_{n\rightarrow \infty} \EE \left[ f\left( \frac{T(0,nx)}{n} \right) \right]  \,=\, \theta^2 f(\wmu (x)) + (1-\theta^2) f(+\infty)\,, $$
where $\theta = \PP[0 \in \C_\infty]$. On one hand, we have
\begin{align}
\label{eq2.1}
\bigg| \EE  \left[ f\left( \frac{T(0,nx)}{n} \right)  \ind{\{ 0\in \C_\infty , nx\in \C_\infty \} }  \right] & - \theta^2 f(\wmu (x))  \bigg| \nonumber \\
& \,\leq \,  \EE \left[ \left|f \left( \frac{T(0,nx)}{n} \right)  - f( \wmu (x))\right|  \ind{\{ 0\in \C_\infty , nx\in \C_\infty \} }\right] \nonumber  \\
& \qquad + |f(\wmu(x))| \times |\theta^2 - \PP[0\in \C_\infty , nx\in \C_\infty]| \,.
\end{align}
Since $f$ is continuous in $\wmu (x)$, for all $\eps >0$, there exists $\eta >0$ such that
\begin{align}
\label{eq2.2}
\EE \Bigg[ \left|f \left( \frac{T(0,nx)}{n} \right)  - f( \wmu (x))\right| & \ind{\{ 0\in \C_\infty , nx\in \C_\infty \} }\Bigg] \nonumber \\
& \,\leq \, \eps \, \PP \left[ \left| \frac{T(0,nx)}{n} - \wmu(x) \right| \leq \eta \,,\,  0\in \C_\infty \,,\, nx\in \C_\infty \right] \nonumber \\
& \qquad + \|f\|_\infty \PP \left[ \left| \frac{T(0,nx)}{n} - \wmu(x) \right| > \eta \,,\,  0\in \C_\infty \,,\, nx\in \C_\infty \right] \nonumber \\
& \,\leq \, \eps +  \|f\|_\infty \PP \left[ \left| \frac{ T^*(0,nx)}{n} - \wmu(x) \right| > \eta \right]\,. 
\end{align}
The convergence of $T^*(0,nx)/n$ towards $\wmu(x)$ in probability implies that the second term of the right hand side of (\ref{eq2.2}) goes to $0$ when $n$ goes to infinity. Moreover, by the FKG inequality, we have
$$  \PP[0\in \C_\infty , nx\in \C_\infty] \,\geq \, \PP[0\in \C_\infty]  \times \PP [ nx\in \C_\infty] \,=\, \theta ^2   \,.$$
Conversely for all fixed $p\in \NN$, for all $n$ large enough, $\B_1 (0,p) \cap \B_1 (nx,p) = \emptyset$ thus
$$ \PP[0\in \C_\infty , nx\in \C_\infty]  \,\geq \, \PP [ 0 \overset{\infty}{\longleftrightarrow} \partial \B_1(0,p) ] \, \PP [ nx \overset{\infty}{\longleftrightarrow} \partial \B_1(nx,p) ]  \,=\, \PP [ 0 \overset{\infty}{\longleftrightarrow} \partial \B_1(0,p) ] ^2 \,.$$
This implies that $\theta^2 \leq \limsup_{n\rightarrow \infty}  \PP[0\in \C_\infty , nx\in \C_\infty] \leq \PP [ 0 \overset{\infty}{\longleftrightarrow} \partial \B_1(0,p) ] ^2$, for all $p$, and sending $p$ to infinity we obtain 
\begin{equation}
\label{eq2.3}
\lim_{n\rightarrow \infty}  \PP[0\in \C_\infty , nx\in \C_\infty]= \theta^2 \,.
\end{equation}
Combining (\ref{eq2.2}) and (\ref{eq2.3}), we get that the right hand side of (\ref{eq2.1}) goes to $0$ as $n$ goes to infinity. On the other hand, we have
\begin{align}
\label{eq2.4}
\bigg| \EE  \left[ f\left( \frac{T(0,nx)}{n} \right)  \ind{\{ 0\notin \C_\infty \} \cup \{ nx\notin \C_\infty \} }  \right] & - \left( 1-\theta^2 \right) f(+\infty)  \bigg| \nonumber \\
& \,\leq\,  \EE \left[ \left|f \left( \frac{T(0,nx)}{n} \right)  - f(+ \infty)\right|  \ind{\{ 0\notin \C_\infty \} \cup \{ nx\notin \C_\infty \} }  \right] \nonumber  \\
& \qquad + |f(+\infty)| \times \left| \PP[\{ 0\notin \C_\infty \} \cup \{ nx\notin \C_\infty \}] -  \left( 1-\theta^2 \right)\right| \nonumber \\
& \,\leq \, 2 \, \|f\|_\infty \PP \left[ \left\{ T(0,nx) \neq +\infty \right\} \bigcap \left( \{ 0\notin \C_\infty \} \cup \{ nx\notin \C_\infty \} \right)  \right] \nonumber \\
& \qquad + \|f\|_\infty \times  |\theta^2 - \PP[0\in \C_\infty , nx\in \C_\infty]| \nonumber \\
& \,\leq \, 2 \, \|f\|_\infty \PP \left[ 0 \notin \C_\infty \,,\, 0 \overset{\infty}{\longleftrightarrow} \partial \B_1(0, n \|x\|_1) \right] \nonumber \\
& \qquad + \|f\|_\infty \times  |\theta^2 - \PP[0\in \C_\infty , nx\in \C_\infty]| \,.
\end{align}
By Theorem \ref{thmfinite}, the first term of the right hand side of (\ref{eq2.4}) vanishes at the limit $n\rightarrow \infty$, and so does the second term by (\ref{eq2.3}). We conclude by combining (\ref{eq2.1}) and (\ref{eq2.4}).
\end{proof}

\begin{rem}
Cox, Durrett \cite{CoxDurrett} and Kesten \cite{Kesten:StFlour} obtain a stronger convergence for $T(0,nx)/n$ even without moment condition, since (for finite passage times) they prove 
$$  \lim_{n\rightarrow \infty} \frac{T(0,nx)}{n} \,=\, \hat \mu (x) \quad \textrm{in probability}\,. $$
We cannot hope to obtain such a convergence. The infinite cluster $\C_\infty$ of finite passage times has holes. Even on the event $\{0\in \C_\infty \}$, the times $T(0,nx)$ are infinite for all $n$ such that $nx \notin \C_\infty$, thus a.s. for an infinite sequence of $n$. They also prove that
$$ \liminf_{n\rightarrow \infty} \frac{T(0,nx)}{n} \,\geq \, \hat \mu (x) \quad \textrm{a.s.} $$
It comes from the fact that in their setting $T(0,nx) \leq \hat T ( 0, nx) = T(S(0), S(nx))$ for $n$ not too small, since $S(0)$ and $S(nx)$ are shells that surround $0$ and $nx$ respectively. We do not have such a comparison between $T$ and $\wT$. However, we can compare $T$ with $T^*$. Indeed, we have
\begin{align*}
\PP \left[ T(0,nx) < T^* (0,nx) \right] &\,\leq\, \PP [ \{ T(0,nx) <\infty \} \cap \left( \{ 0 \notin \C_\infty \} \cup \{ nx \notin \C_\infty \} \right)] \\
& \,\leq \, \PP \left[ 0 \notin \C_\infty \,,\, 0\overset{\infty}{\longleftrightarrow} nx \right] \\
& \, \leq \, A_4 e^{-A_5 n \|x\|_1}
\end{align*}
by Theorem \ref{thmfinite}. By an application of Borel Cantelli's Lemma, we obtain that 
$$ \textrm{a.s., for all }n \textrm{ large enough,} \quad T(0,nx) \,\geq \,T^* (0,nx) \,. $$ 
Nevertheless, since we do not have an a.s. convergence for $T^*(0,nx)/n$, we do not get an a.s. lower bound for $\liminf_{n\rightarrow \infty} T(0,nx)/n$.
\end{rem}


\section{Shape theorem}
\label{secshape}

\subsection{Strong shape theorem for $\wB_t$}
\label{secwB}

In this section we suppose that $F([0,+\infty[) > p_c(d)$, and for the proof of Theorem \ref{thmwB} (ii) we also suppose that $\wmu (e_1) >0$, thus $\wmu$ is a norm. We emphasize that in section \ref{secpos}, during the study of the positivity of $\wmu$ (proof of Theorem \ref{thmpositivity}), we will use the fact that the convergence of $\wB_t$ towards $\B_{\tilde \mu}$ is proved as soon as $F([0,+\infty[) > p_c(d)$ and $\wmu (e_1) > 0$ (without any assumption on $F(\{0\})$ a priori).

The proof of Theorem \ref{thmwB} is classic since the times $\wT$ have good moment properties. We could mimic the proof of Theorem 3.1 in \cite{Kesten:StFlour}. We rather decide to follow the steps of Garet and Marchand's proof in \cite{GaretMarchand04} (see also \cite{Scholler:these}). We need some preliminary lemmas.

\begin{lem}
\label{lem1}
There exists a constant $C_4$ such that for any $\eps >0$, a.s., there exists $R>0$ such that for all $x,y\in \ZZ^d$ we have
$$ \left. \begin{array}{r} \|x\|_1\, \geq\, R \\ \|x- y\|_1 \,\leq\, \eps \|x\|_1 \end{array} \right\} \Longrightarrow \wT (x,y) \,\leq\, C_4 \eps \|x\|_1 \,. $$
\end{lem}

\begin{proof}
Fix $\eps >0$. Let $C_3$ the constant given in Proposition \ref{prop1}. For all $m\in \NN^*$, we define the event $E_m$ by
$$ E_m \,=\, \left\{ \exists x,y\in \ZZ^d \,\textrm{ s.t. }\, \|x\|_1 \geq m \,,\, \|x-y\|_1 \leq \eps \|x\|_1  \,,\,  \wT(x,y)> C_3 \eps \|x\|_1\right\} \,. $$
Then we have
\begin{align}
\label{eqjust}
\PP[E_m] & \,\leq \, \sum_{x\in \ZZ^d, \|x\|_1 \geq m}\,\,\, \sum_{y \in \ZZ^d \cap B_1 (x, \eps \|x\|_1)} \PP [\wT (x,y) > C_3\eps \|x\|_1] \nonumber \\
& \,\leq \, \sum_{x\in \ZZ^d, \|x\|_1 \geq m}  \,\,\, \sum_{z\in \ZZ^d , \|z\|_1 \leq \eps \|x\|_1} \PP [\wT(0,z) > C_3 \eps \|x\|_1] \nonumber \\
& \,\leq \,  \sum_{x\in \ZZ^d, \|x\|_1 \geq m}\,\,\, \sum_{z\in \ZZ^d , \|z\|_1 \leq \eps \|x\|_1} C_1 e^{-C_2 C_3 \eps \|x\|_1} \\
& \,\leq \,  \sum_{x\in \ZZ^d, \|x\|_1 \geq m} C'_1 \|x\|_1^d  e^{-C_2 C_3 \eps \|x\|_1} \nonumber \\
& \,\leq\, \sum_{k\geq m} C_1' C_2' k^d e^{- C_2 C_3 \eps k} \,, \nonumber
\end{align}
where $C'_i$ are constants (depending on the dimension $d$, on $\eps$ and on $F$), and Inequality (\ref{eqjust}) comes from Proposition \ref{prop1}. Thus $\sum_m \PP[E_m] < \infty$, and a simple application of Borel Cantelli's Lemma concludes the proof.
\end{proof}

\begin{lem}
\label{lem2}
For all $\eps >0$, a.s. there exists $R>0$ such that for all $y\in \ZZ^d$
$$\|y\|_1 \,\geq \, R \,\Longrightarrow  \frac{ \left| \wT (0,y) - \wmu (y) \right|}{\|y\|_1} \,\leq\, \eps \,. $$
\end{lem}

\begin{proof}
If Lemma \ref{lem2} is wrong, there exists $\eps >0$ and a sequence $(y_n) $ of vertices in $\ZZ^d$ such that $\lim_{n\rightarrow \infty} \|y_n\|_1 = + \infty$ and for all $n\in \NN$ we have $|\wT (0,y_n) - \wmu (y_n) | >\eps \|y_n\|_1$. Up to extraction, we can suppose that $(y_n / \|y_n\|_1)_{n\in \NN}$ converges to a point $z \in \RR^d$ such that $\|z\|_1 = 1$. Let $z'$ be an approximation of $z$ on $\ZZ^d$ in the following sense. Fix $ \hat \eps$ to be chosen later. There exists $z' \in \ZZ^d$ (with $\|z'\|_1$ large enough) such that $\left\| z - z'/\|z'\|_1 \right\|_1 \leq \hat \eps$
(see Figure \ref{fig1}).
\begin{figure}[h!]
\centering
\begin{picture}(0,0)%
\includegraphics{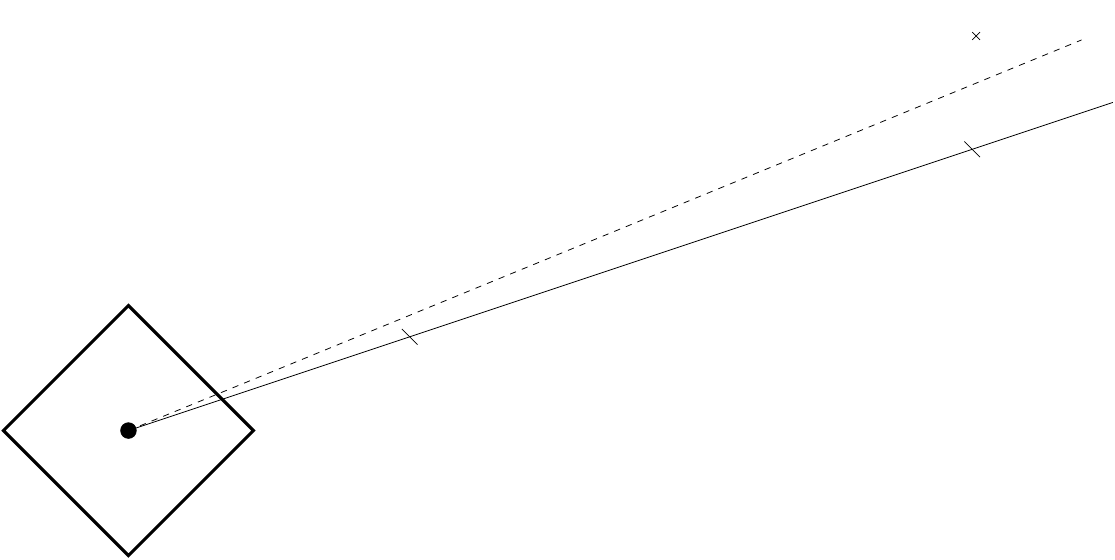}%
\end{picture}%
\setlength{\unitlength}{1973sp}%
\begingroup\makeatletter\ifx\SetFigFont\undefined%
\gdef\SetFigFont#1#2#3#4#5{%
  \reset@font\fontsize{#1}{#2pt}%
  \fontfamily{#3}\fontseries{#4}\fontshape{#5}%
  \selectfont}%
\fi\endgroup%
\begin{picture}(10695,5364)(1768,-7594)
\put(10951,-2461){\makebox(0,0)[rb]{\smash{{\SetFigFont{8}{9.6}{\rmdefault}{\mddefault}{\updefault}{\color[rgb]{0,0,0}$y_n$}
}}}}
\put(2926,-6736){\makebox(0,0)[rb]{\smash{{\SetFigFont{8}{9.6}{\rmdefault}{\mddefault}{\updefault}{\color[rgb]{0,0,0}$0$}
}}}}
\put(4201,-6211){\makebox(0,0)[lb]{\smash{{\SetFigFont{8}{9.6}{\rmdefault}{\mddefault}{\updefault}{\color[rgb]{0,0,0}$z'/\|z'\|_1$}
}}}}
\put(2551,-5461){\makebox(0,0)[rb]{\smash{{\SetFigFont{8}{9.6}{\rmdefault}{\mddefault}{\updefault}{\color[rgb]{0,0,0}$\partial B_1(0,1)$}
}}}}
\put(11551,-4186){\makebox(0,0)[b]{\smash{{\SetFigFont{8}{9.6}{\rmdefault}{\mddefault}{\updefault}{\color[rgb]{0,0,0}$y_n' = \left\lfloor \frac{\|y_n\|_1}{\|z'\|_1} \right\rfloor z'$}
}}}}
\put(6001,-5836){\makebox(0,0)[b]{\smash{{\SetFigFont{8}{9.6}{\rmdefault}{\mddefault}{\updefault}{\color[rgb]{0,0,0}$z'$}
}}}}
\put(3901,-5761){\makebox(0,0)[b]{\smash{{\SetFigFont{8}{9.6}{\rmdefault}{\mddefault}{\updefault}{\color[rgb]{0,0,0}$z$}
}}}}
\end{picture}
\caption{$z$ is the limit of $y_n / \|y_n\|_1$, $z'/\|z'\|_1$ is a rational approximation of $z$, and $y_n'= h_n z'$ is an approximation of $y_n$ along $\ZZ z'$.}
\label{fig1}
\end{figure}
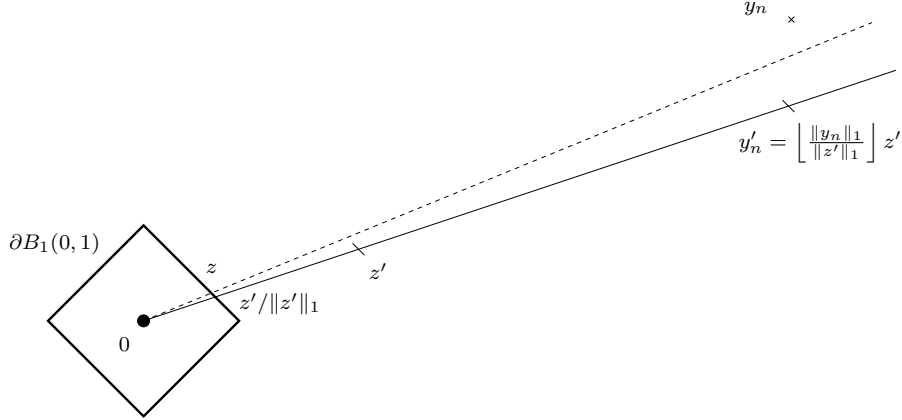
For all $n\in \NN$, let $h_n =  \left\lfloor \|y_n\|_1 / \|z'\|_1 \right\rfloor $, and $y_n' = h_n z' $. Then
\begin{align*}
\left\| y_n - y_n' \right\|_1 & \,\leq\, \left\| y_n - \frac{\|y_n\|_1}{\|z'\|_1} z' \right\|_1 + \left|  \frac{\|y_n\|_1}{\|z'\|_1} -   h_n \right| \times \|z'\|_1\\
& \,\leq\, \left\| \frac{y_n}{\|y_n\|_1} - \frac{z'}{\|z'\|_1} \right\|_1 \times \|y_n\|_1 + \|z'\|_1\\
& \, \leq \, \left( \left\| \frac{y_n}{\|y_n\|_1} - z \right\|_1 + \left\|  \frac{z'}{\|z'\|_1} -z \right\|_1 \right)\times \|y_n\|_1 + \|z'\|_1 \\
& \,\leq\, 3\, \hat \eps \,\|y_n\|_1 \,,
\end{align*}
where the last inequality holds for all large $n$ since $\lim_{n\rightarrow \infty} \|y_n\|_1 = +\infty$ and $\lim_{n\rightarrow \infty} y_n / \|y_n \|_1 = z$. By Lemma \ref{lem1} we obtain that a.s., for all $n$ large enough, $\wT (y_n, y_n' ) \leq 3 \hat \eps  C_4 \|y_n\|_1$. We have the following upper bound :
\begin{align}
\label{eqjust2}
\frac{ \left| \wT (0,y_n) - \wmu (y_n) \right|}{\|y_n\|_1} & \,\leq \, \frac{ \left| \wT (0,y_n) - \wT(0,y'_n) \right|}{\|y_n\|_1} + \left|  \frac{\wT (0,y_n')}{\|y_n\|_1}  - \wmu \left( \frac{z'}{\|z'\|_1} \right) \right|+  \left|\wmu \left( \frac{z'}{\|z'\|_1} \right)- \wmu\left( \frac{y_n }{\| y_n\|_1}\right) \right| \nonumber \\
& \,\leq \, \frac{\wT (y_n, y_n')}{\| y_n\|_1} +  \left| \frac{h_n \|z'\|_1}{\|y_n\|_1} \times \frac{\wT ( 0,h_n z')}{h_n \|z'\|_1}  - \wmu \left( \frac{z'}{\|z'\|_1} \right) \right|  + \wmu(e_1) \left\|  \frac{z'}{\|z'\|_1} - \frac{y_n }{\| y_n\|_1} \right\|_1\,.
\end{align}
We just proved that a.s., for all $n$ large enough, the first term in (\ref{eqjust2}) is smaller than $3 \hat \eps C_4$. Moreover,
$$ \left\|  \frac{z'}{\|z'\|_1} - \frac{y_n }{\| y_n\|_1} \right\|_1  \,\leq\,  \left\| \frac{ z'}{\|z'\|_1 }- z \right\|_1 + \left\| z - \frac{y_n }{\| y_n\|_1 } \right\|_1  \,,$$
thus for $n$ large enough, the third term of (\ref{eqjust2}) is smaller than $2 \wmu(e_1) \hat \eps$. By Theorem \ref{thmwT}, a.s., for all $u\in \ZZ^d$ we have $\lim_{p\rightarrow \infty} \wT(0,pu)/p = \wmu (u)$. In particular, since $\lim_{n\rightarrow \infty} h_n =+\infty$, we have a.s., whatever $z'\in \ZZ^d$, the convergence 
$$\lim_{n\rightarrow \infty} \frac{\wT \left( 0,h_n z'\right)}{h_n \|z'\|_1} \,=\, \frac{\wmu(z')}{\|z'\|_1} \,=\,  \wmu \left( \frac{z'}{\|z'\|_1 } \right)\,.$$
Since $\lim_{n\rightarrow \infty} h_n / \|y_n\|_1 = 1$, we get that a.s., for all $n$ large enough, the second term in (\ref{eqjust2}) is smaller than $\hat\eps$. Choosing $\hat \eps$ small compared to $\eps$, we get that a.s., for all $n$ large enough,
$$ \frac{ \left| \wT (0,y_n) - \wmu (y_n) \right|}{\|y_n\|_1}  \,\leq \, (3  C_4 +  1+ 2 \wmu(e_1) ) \hat \eps \,\leq \, \eps \,.$$
This is a contradiction, thus Lemma \ref{lem2} is proved.
\end{proof}

\begin{proof}[Proof of Theorem \ref{thmwB}]
Part (i) of Theorem \ref{thmwB} is a straightforward consequence of Lemma \ref{lem2}. It remains to prove part (ii), under the hypotheses that $F([0,+\infty[)>p_c(d)$ and that $\wmu$ is a norm. The proof of Theorem \ref{thmwB} is completed by combining this result with Theorem \ref{thmpositivity} (which is proved in section \ref{secpos}).

Suppose that part (ii) is wrong because of the second inclusion : there exists $\eps >0$ and a sequence $t_n \in \RR^+$ such that $\lim_{n\rightarrow \infty} t_n = +\infty$ and for all $n$, $\wB_{t_n}/t_n$ is not included in $(1+\eps) \B_{\tilde \mu}$. Thus there exists a sequence of points $x_n \in \RR^d$ such that for all $n$, $t_n x_n \in \wB_{t_n}$, and $\wmu (x_n) > 1+\eps$. Since $t_n x_n \in \wB_{t_n}$, there exists $y_n \in \ZZ^d$ such that $\wT (0,y_n) \leq t_n$ and $t_nx_n \in \{ y_n +u \,|\, u \in [-1/2,1/2]^d \}$. We know that $\wmu(x_n) \leq \|x_n\|_1 \wmu (e_1)$, thus $\|x_n\|_1 \geq (1+\eps) / \wmu (e_1)  >0$, so $\lim_{n\rightarrow \infty} t_n =+\infty$ implies $\lim_{n\rightarrow \infty} \|y_n\|_1 = +\infty$. Moreover $|\wmu (y_n) - \wmu(t_n x_n)| \leq \wmu (e_1) \| y_n - t_n x_n \|_1 \leq C$, a constant depending only on the dimension and on $F$. This implies that for $n$ large enough
$$ \wmu (y_n) \,\geq\, t_n \wmu (x_n) - C \,\geq\, t_n (1+\eps) - C \,\geq\, t_n (1+\eps/2) \,,$$
thus
$$ \frac{\wmu (y_n) - \wT(0,y_n)}{\|y_n\|_1} \,\geq \,\left( 1 - \frac{1}{1+\eps/2} \right) \wmu \left( \frac{y_n}{\|y_n\|_1} \right) \,\geq \, \left( 1 - \frac{1}{1+\eps/2} \right)  \min_{z\in \RR^d\,,\,\|z\|_1 = 1} \wmu (z) \,>\,0 \,.$$
According to Lemma \ref{lem2}, a.s., this cannot happen for a sequence of vertices $y_n \in \ZZ^d$ such that $\lim_{n\rightarrow \infty}\|y_n\|_1 = +\infty$.

Suppose that part (ii) is wrong because of the first inclusion : there exists $\eps >0$ and a sequence $t_n \in \RR^+$ such that $\lim_{n\rightarrow \infty} t_n = +\infty$ and for all $n$, $(1-\eps) \B_{\tilde \mu}$ is not included in $\wB_{t_n}/t_n$. Then there exists a sequence of points $x_n\in \RR^d$ such that $\wmu (x_n) \leq 1-\eps$ and $t_n x_n \notin \wB_{t_n}$ for all $n$. Let $y_n \in \ZZ^d$ be such that $t_n x_n \in \{ y_n +u \,|\, u \in [-1/2,1/2]^d \}$: then $\wT (0,y_n) > t_n$. This implies that $\lim_{n\rightarrow \infty} \|y_n\|_1 = +\infty$, otherwise up to extraction we can suppose that $y_n$ converges to a point $z$, and since $y_n \in \ZZ^d$ it means that $y_n = z$ for all large $n$, thus $\wT(0,y_n) = \wT (0,z)$ which cannot be bigger than $t_n$ for all $n$ since $\wT(0,z)$ is finite a.s. Moreover $ \wmu (y_n) \leq  \wmu (t_n x_n) + C \leq t_n (1-\eps/2) $ for all large $n$, thus
$$ \frac{ \wT(0,y_n)- \wmu (y_n)}{\|y_n\|_1} \,\geq \,\left(  \frac{1}{1-\eps/2} -1 \right) \wmu \left( \frac{y_n}{\|y_n\|_1} \right) \,\geq \, \left(  \frac{1}{1-\eps/2} -1 \right)  \min_{z\in \RR^d\,,\,\|z\|_1 = 1} \wmu (z) \,>\,0 \,. $$
According to Lemma \ref{lem2}, a.s. this cannot happen for a sequence of vertices $y_n \in \ZZ^d$ such that $\lim_{n\rightarrow \infty}\|y_n\|_1 = +\infty$. This ends the proof of Theorem \ref{thmwB}, up to the question of the positivity of $\wmu$ (which is handled through Theorem \ref{thmpositivity}).
\end{proof}


\subsection{Weak shape theorem for $B_t^*$ and $B_t$}
\label{secB}

In this section, we admit Theorem \ref{thmpositivity}. We suppose that $F([0,+\infty[) >p_c(d)$ and $F(\{ 0 \}) <p_c(d)$. By Theorem \ref{thmpositivity}, we conclude that $\wmu(e_1) >0$, thus $\wmu$ is a norm, and  by Theorem \ref{thmwB} the strong shape theorem holds for $\wB_t$. We recall that $|A|$ denotes the cardinality of a set $A$, and $A \triangle B$ is the symmetric difference between two sets $A$ and $B$. For all $t\geq 0$, we define the set of vertices $B_t^{*,v}$ in the same way as $B_t^v$ :
$$ B_t^{*,v} \,=\, \{ z\in \ZZ^d \,|\, T^*(0,z) \leq t \} \,. $$
\begin{prop}
\label{propcard}
Suppose that $F([0,+\infty[) > p_c(d)$ and $F(\{0\}) <p_c(d)$. Then we have
$$  \lim_{t\rightarrow \infty} \frac{1}{t^d} \left| B_t^{*,v} \triangle ( t \B_{\tilde \mu} \cap \ZZ^d) \right| \,=\, 0 \quad \textrm{a.s.}$$
\end{prop}

\begin{proof}[Proof of Proposition \ref{propcard}]
For this proof, we follow \cite{CoxDurrett} and \cite{Kesten:StFlour}. We recall that $B_t^{*,v}$ denotes the set $B_t^* \cap \ZZ^d$, i.e., $B_t^{*,v} = \{ x\in \ZZ^d \,|\, T^*(0,x) \leq t \}$. Let $A^c = \RR^d \smallsetminus A$. Fix $\eps >0$. We have the following inclusion :
\begin{align}
\label{eqinclusions}
B_t^{*,v}  \triangle (t \B_{\tilde \mu} \cap \ZZ^d) \,\subset \, &  \left[ \left( t(1+\eps) \B_{\tilde \mu}  \smallsetminus t (1-\eps) \B_{\tilde \mu} \right) \cap \ZZ^d \right] \bigcup \left[ \left( t(1-\eps) \B_{\tilde \mu} \cap \ZZ^d \right) \cap (B_t^{*,v})^c \right]  \nonumber \\
& \qquad \bigcup\left[ B_t ^{*,v} \cap (t(1+\eps) \B_{\tilde \mu})^c \right] \,.
\end{align}
We study the three terms appearing in the right hand side of (\ref{eqinclusions}) separately. 

Notice that for any $A\subset \RR^d$, we can bound the cardinality $|A\cap \ZZ^d|$ of $A\cap \ZZ^d$ by the volume $\L^d( \{ a+u \,|\, a\in A\,,\, u\in [-1/2, 1/2]^d \} )$. Thus for $t$ large enough we have
\begin{align}
\label{eqterm1}
\frac{1}{t^d} \bigg| \left( t(1+\eps) \B_{\tilde \mu}  \smallsetminus  t (1-\eps) \B_{\tilde \mu} \right) &  \cap \ZZ^d \bigg| \nonumber \\
& \, \leq\, \frac{1}{t^d} \L^d  \left( \left\{ x+u \,|\, x\in \left( t(1+\eps) \B_{\tilde \mu}  \smallsetminus t (1-\eps) \B_{\tilde \mu} \right) \,,\, u \in [-1/2, 1/2]^d \right\} \right)\nonumber \\
& \,\leq\,\frac{1}{t^d} \L^d \left(  t(1+2\eps) \B_{\tilde \mu}  \smallsetminus t (1-2\eps) \B_{\tilde \mu}  \right) \nonumber \\
& \,\leq\,  g(\eps) \L^d(\B_{\tilde \mu}) \,, 
\end{align}
where $g(\eps) = (1+2\eps)^d - (1-2\eps)^d$ goes to $0$ when $\eps$ goes to $0$.

Looking at the second term in the right hand side of (\ref{eqinclusions}), we notice that
$$ \left[ \left( t(1-\eps) \B_{\tilde \mu} \cap \ZZ^d \right) \cap (B_t^{*,v})^c \right]  \,=\, \{ x\in \ZZ^d \,|\,  \wmu(x) \leq t(1-\eps) \,,\, T^*(0,x)>t \} $$
Since $\wmu$ is a norm, the set $\B_{\tilde \mu}$ is compact, thus there exists a constant $K$ such that $\B_{\tilde \mu} \subset [-K,K]^d$. By Theorem \ref{thmwB} and the convexity of $\B_{\tilde \mu}$, we know that a.s., for all $t$ large enough, $t(1-\eps) \B_{\tilde \mu} \subset t (1-\eps/2) \B_{\tilde \mu}\subset \wB_{t(1-\eps/2)}$, thus
\begin{align*}
 \left[ \left( t(1-\eps) \B_{\tilde \mu} \cap \ZZ^d \right) \cap (B_t^{*,v})^c \right] & \, \subset \, \{ x\in \ZZ^d \cap [-Kt,Kt]^d \,|\,  \wT(0,x) \leq t(1-\eps/2) \,,\,  T^*(0,x)>t \} \\
& \,\subset \, \{ x\in \ZZ^d \cap [-Kt,Kt]^d \,|\,  T(0^*,\w0) + T(x^*,\wx) \geq t \eps/2  \}\,.
\end{align*}
We have $T(0^*,\w0) <\infty$ since $0^*$ and $\w0$ belong to $\C_\infty$, thus $T(0^*,\w0) \leq t \eps/4 $ for all large $t$. We obtain that a.s., for all $J>0$, for all $t$ large enough,
\begin{align*}
\frac{1}{t^d} \left| \left( t(1-\eps) \B_{\tilde \mu} \cap \ZZ^d \right) \cap (B_t^{*,v})^c \right| & \,\leq\, \frac{1}{t^d} \sum_{x\in [-Kt,Kt]^d \cap \ZZ^d} \ind{\{ T(x^*,\tilde x) > t\eps/4 \}}  \\
& \,\leq\, \frac{1}{t^d} \sum_{x\in [-Kt,Kt]^d \cap \ZZ^d} \ind{\{  T(x^*,\tilde x) > J \}} \,.
\end{align*}
Applying the ergodic theorem, we obtain that 
$$\lim_{t\rightarrow \infty} \frac{1}{|[-Kt,Kt]^d \cap \ZZ^d|} \sum_{x\in [-Kt,Kt]^d \cap \ZZ^d} \ind{\{  T(x^*,\tilde x) > J \}} \,=\, \psi_J \quad \textrm{a.s. and in }L^1\,,$$
where $\psi_J$ is a random variable satisfying 
$$\EE [\psi_J] = \PP [ T(0^*,\w0) > J] \underset{J\rightarrow \infty}{\longrightarrow} \PP[ T(0^*,\w0) = +\infty ] = 0 \,.$$
We conclude that 
\begin{equation}
\label{eqterm2}
\psi \,=\, \limsup_{t\rightarrow \infty }\frac{1}{t^d} \left| \left( t(1-\eps) \B_{\tilde \mu} \cap \ZZ^d \right) \cap (B_t^{*,v})^c \right|
\end{equation}
is a non negative random variable satisfying $\EE [\psi] \leq \EE[\psi_J]$, for all $J$, thus $\psi = 0$ a.s.

To study the third term of the right hand side of $(\ref{eqinclusions})$ in a similar way, we need an additional argument to prove that $B_t^{*,v} / t$ is included in a compact set. Let $k \in ]0, +\infty[$ be a positive and finite real number, and define the variables
$$  t'(e) \,=\, \left\{ \begin{array}{ll} t(e) & \textrm{ if } t(e) < k\\ k & \textrm{ if } k\leq t(e) < +\infty \\ +\infty & \textrm{ if } t(e) = +\infty \,. \end{array} \right. $$
Then $t'(e) \leq t(e)$ for all $e\in \EE^d$. Choose an $M\geq k$ in the definition of the regularized times $\wT'$ associated with the variables $(t'(e))$. Since $\C'_\infty$, the infinite cluster of the percolation $(\ind{\{ t'(e) < \infty \}}, e\in \EE^d)$, is equal both to $\C'_M$ - the infinite cluster of the percolation $(\ind{\{ t'(e) \leq M \}}, e\in \EE^d)$ - and to $\C_\infty$, we have for all $x\in \ZZ^d$ the equality $x^* = x^{* '} = \wx '$, where $x^{*'}$ (resp. $\wx '$) denotes the point of $\C'_\infty$ (resp. $\C'_M$) which is the closest to $x$ for the norm $\| \cdot \|_1$ (with a deterministic rule to break ties). We obtain 
$$ T^* (0,x) \,=\, T( 0^*, x^*) \,\geq \, T' (0^*, x^*) \,=\, T' (\w0', \wx') \,=\, \wT'(0,x) \,.$$
Let $\wmu'(x) = \lim_{n\rightarrow \infty} \wT'(0,nx)/n$, and extend $\wmu'$ to $\RR^d$. Since the passage times $t'(e)$ satisfy $\PP[t'(e) <\infty] = F([0,+\infty[)>p_c(d)$ and $\PP[t'(e) =0] = F(\{0\}) <p_c(d)$, then $\wmu'$ is a norm by Theorem \ref{thmpositivity}, thus by Theorem \ref{thmwB} a strong shape theorem holds for 
$$\wB_t'  \,=\, \{ x+u \,|\, x\in \ZZ^d \,,\, \wT'(0,x)\leq t \,,\, u\in [-1/2,1/2]^d \}\,.$$
The limit shape $\B_{\tilde \mu'}$ is compact, thus there exists $K' \in \RR^+$ such that $\B_{\tilde \mu'} \subset [-K', K']^d$. In particular,
$$ \textrm{a.s., for all } t\textrm{ large enough,} \quad B^*_t \,\subset \, \wB'_t \,\subset\, 2 t \B_{\tilde \mu'}\,\subset\,  [-2K't, 2K't]^d \,.$$
We now proceed as in the study of the second term of (\ref{eqinclusions}). A.s., for all $t$ large enough, we have
\begin{align*}
 \left[ B_t^{*,v}  \cap \left( t(1+\eps) \B_{\tilde \mu} \right)^c \right] & \, \subset \, \{ x\in \ZZ^d \cap [-2K't,2K't]^d \,|\, T^*(0,x) \leq t  \,,\,  \wmu(x) > t(1+\eps) \} \\
& \,\subset \, \{ x\in \ZZ^d \cap [-2K't,2K't]^d \,|\, T^*(0,x) \leq t\,,\, \wT(0,x) > t(1+\eps/2) \} \\
& \,\subset \, \{ x\in \ZZ^d \cap [-2K't,2K't]^d \,|\,  T(x^*,\wx) >t\eps/4 \} \,. \\
\end{align*}
A.s., for all $t$ large enough, we obtain
\begin{equation}
\label{eqterm3}
\frac{1}{t^d} \left| B_t^{*,v}  \cap \left( t(1+\eps) \B_{\tilde \mu} \right)^c \right|  \,\leq\, \frac{1}{t^d} \sum_{x\in [-2K't,2K't]^d \cap \ZZ^d} \ind{\{  T(x^*,\tilde x) > t\eps/4 \}} \,,
\end{equation}
that goes a.s. to $0$ as t goes to infinity as proved in (\ref{eqterm2}). Proposition \ref{propcard} is proved by combining (\ref{eqinclusions}), (\ref{eqterm1}), (\ref{eqterm2}) and (\ref{eqterm3}).
\end{proof}

\begin{proof}[Proof of Theorem \ref{thmB} (i)]
We define a discrete approximation of $\B_{\tilde \mu}$ with cubes of side length $1/t$ :
\begin{equation}
\label{eqBmut}
 \B_{\tilde \mu , t} \,=\, \{ y+u \,|\, y \in \B_{\tilde \mu} \cap \ZZ^d /t \,,\, u \in [-(2t)^{-1} , (2t)^{-1} [\} \,.
 \end{equation}
All the norms are equivalent in $\RR^d$, thus for any fixed $\eps >0$, for all $t $ large enough we have
$$\L^d \big( \B_{\tilde \mu} \triangle \B_{\tilde \mu , t}  \big)  \,\leq \, \L^d \left( (1+\eps) \B_{\tilde \mu} \smallsetminus (1-\eps) \B_{\tilde \mu} \right) \, \leq \, \hat g(\eps) \L^d (\B_{\tilde \mu}) \,,$$
where $\hat g(\eps) = (1+\eps)^d - (1-\eps)^d$ goes to $0$ when $\eps$ goes to $0$. Moreover we have
$$ \L^d \left(  \frac{B_t^*}{t} \triangle \B_{\tilde \mu , t} \right) \,=\, \frac{1}{t^d} \left| B_t^{*,v} \triangle ( t \B_{\tilde \mu} \cap \ZZ^d) \right|   $$
that goes to $0$ a.s. by Proposition \ref{propcard}. This concludes the proof of Theorem \ref{thmB} (i).
\end{proof}

\begin{proof}[Proof of Theorem \ref{thmB} (ii)]
Let 
$$m_t \,=\, \frac{1}{t^d} \sum_{x\in B_t^v} \delta_{x/t}\,.$$
We want to prove that on $\{0\in \C_\infty \}$, a.s., for every continuous bounded real function $f$ defined on $\RR^d$, we have
$$ \lim_{t\rightarrow \infty} \int_{\RR^d} f \, dm_t  \,=\, \lim_{t\rightarrow \infty} \frac{1}{t^d} \sum_{x\in B_t^v} f\left( \frac{x}{t} \right) \,=\, \theta \int_{\B_{\tilde \mu}} f \, d\L^d \,.$$
Let us give a short sketch of the proof. In a first step, we use Proposition \ref{propcard} to transform the sum of $f(x/t)$ over $x\in B_t^v$ in $\int_{\RR^d} f \, dm_t$ into a sum over $x\in \B_{\tilde \mu} \cap \C_\infty$. We divide $\B_{\tilde \mu}$ into small cubes $D(y,t_0)$ of center $y \in \B_{\tilde \mu} \cap \ZZ^d /t_0$ and side length $1/t_0$ (see Figure \ref{fig2}). Since $f$ is continuous, thus uniformly continuous on compact sets, for $t_0$ large enough $f$ is almost constant on each such cube. Thus considering a $t \geq t_0$, up to a small error, we can replace every $f(x/t)$ for $x/t \in D(y,t_0)$ by the same constant $f(y)$. This term $f(y)$ is multiplied in $\int_{\RR^d} f \, dm_t$ by the proportion of $x\in t D(y,t_0)$ that belong to $\C_\infty$. By an application of the ergodic theorem, this proportion goes to $\theta$ when $t$ goes to infinity for a fixed cube $D(y,t_0)$.
\begin{figure}[h!]
\centering
\begin{picture}(0,0)%
\includegraphics{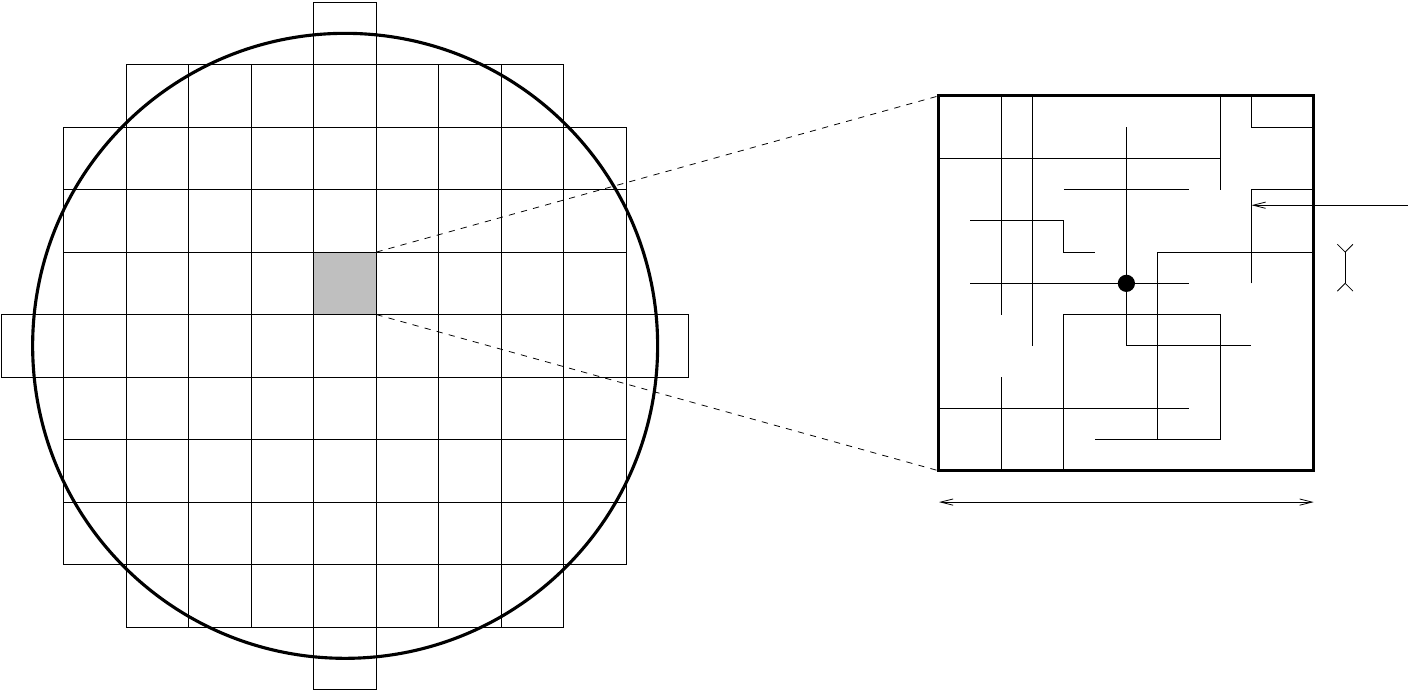}%
\end{picture}%
\setlength{\unitlength}{1973sp}%
\begingroup\makeatletter\ifx\SetFigFont\undefined%
\gdef\SetFigFont#1#2#3#4#5{%
  \reset@font\fontsize{#1}{#2pt}%
  \fontfamily{#3}\fontseries{#4}\fontshape{#5}%
  \selectfont}%
\fi\endgroup%
\begin{picture}(13677,6624)(1189,-8173)
\put(3451,-1786){\makebox(0,0)[rb]{\smash{{\SetFigFont{8}{9.6}{\rmdefault}{\mddefault}{\updefault}{\color[rgb]{0,0,0}$\B_{\wmu}$}%
}}}}
\put(12001,-6811){\makebox(0,0)[b]{\smash{{\SetFigFont{8}{9.6}{\rmdefault}{\mddefault}{\updefault}{\color[rgb]{0,0,0}$1/t_0$}%
}}}}
\put(11851,-4111){\makebox(0,0)[b]{\smash{{\SetFigFont{8}{9.6}{\rmdefault}{\mddefault}{\updefault}{\color[rgb]{0,0,0}$y$}%
}}}}
\put(10651,-2236){\makebox(0,0)[b]{\smash{{\SetFigFont{8}{9.6}{\rmdefault}{\mddefault}{\updefault}{\color[rgb]{0,0,0}$D(y,t_0)$}%
}}}}
\put(14851,-3661){\makebox(0,0)[lb]{\smash{{\SetFigFont{8}{9.6}{\rmdefault}{\mddefault}{\updefault}{\color[rgb]{0,0,0}$\C_\infty / t$}%
}}}}
\put(14251,-4261){\makebox(0,0)[lb]{\smash{{\SetFigFont{8}{9.6}{\rmdefault}{\mddefault}{\updefault}{\color[rgb]{0,0,0}1/t}%
}}}}
\end{picture}%
\caption{The set $\B_{\tilde \mu}$ and its discrete approximation by cubes of the form $D(y,t_0)$ for $y\in \B_{\tilde \mu} \cap \frac{\ZZ^d}{t_0}$.}
\label{fig2}
\end{figure}
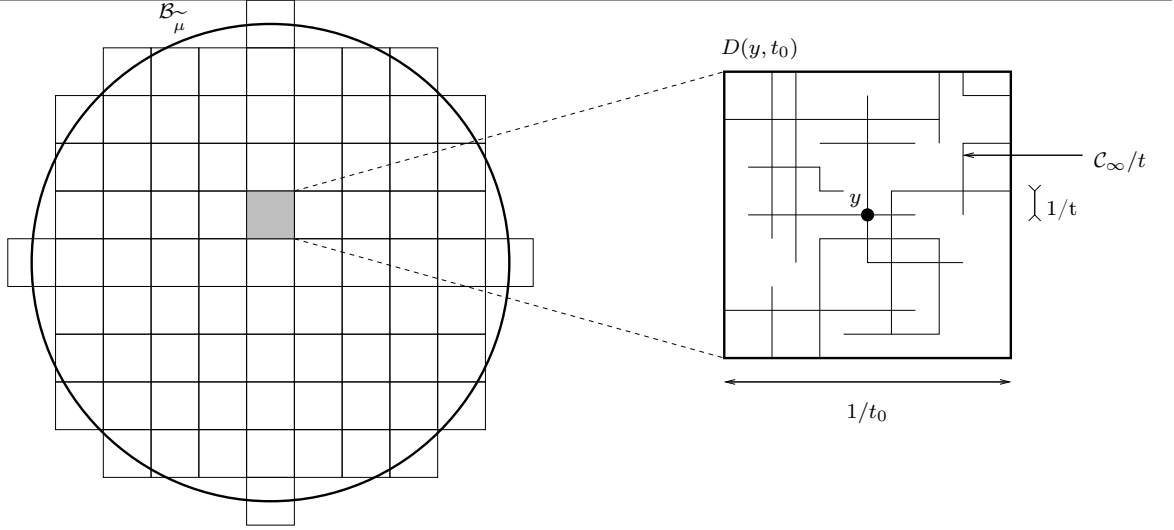

We now start the proof. Let $f$ be a continuous bounded real function defined on $\RR^d$, and let us denote by $\|f\|_\infty$ the supremum of $f$. We have
\begin{align}
\label{eqrassemble}
\left| \int_{\RR^d} f \, dm_t -  \theta \int_{\B_{\tilde \mu}} f \, d\L^d \right| &  \,\leq\, \left| \int_{\RR^d} f \, dm_t - \frac{1}{t^d} \sum_{x\in t\B_{\tilde \mu} \cap \C_\infty} f \left( \frac{x}{t} \right) \right| \nonumber \\
& \qquad + \Bigg| \frac{1}{t^d} \sum_{x\in t\B_{\tilde \mu} \cap \C_\infty} f \left( \frac{x}{t} \right)   -  \frac{1}{t_0^d} \sum_{y \in \B_{\tilde \mu}\cap \frac{\ZZ^d}{t_0}}  f(y)  \left[ \frac{t_0^d}{t^d} \sum_{x\in tD(y,t_0) \cap \ZZ^d}   \ind{\{ x\in \C_\infty \}} \right] \Bigg| \nonumber \\
& \qquad + \Bigg|  \frac{1}{t_0^d} \sum_{y \in \B_{\tilde \mu}\cap \frac{\ZZ^d}{t_0}}  f(y)  \left[ \frac{t_0^d}{t^d} \sum_{x\in tD(y,t_0) \cap \ZZ^d}   \ind{\{ x\in \C_\infty \}} \right] -  \frac{\theta }{t_0^d} \sum_{y \in \B_{\tilde \mu} \cap \frac{\ZZ^d}{t_0}} f(y) \Bigg| \nonumber\\
&  \qquad + \theta \bigg|  \frac{1}{t_0^d} \sum_{y \in \B_{\tilde \mu} \cap \frac{\ZZ^d}{t_0}} f(y) - \int_{\B_{\tilde \mu}} f \, d\L^d \bigg| \,.
\end{align}
We study these four terms separately. We have
\begin{equation}
\label{eqetape1}
\left| \int_{\RR^d} f \, dm_t - \frac{1}{t^d} \sum_{x\in t\B_{\tilde \mu} \cap \C_\infty} f \left( \frac{x}{t} \right) \right|  \,\leq\, \| f \|_\infty \frac{\left| B_t^v \triangle \left( t\B_{\tilde \mu} \cap \C_\infty \right) \right| }{t^d} \,.
\end{equation}
By Proposition \ref{propcard}, we know that a.s., $\lim_{t\rightarrow \infty} | B_t^{*,v} \triangle (t\B_{\tilde \mu} \cap \ZZ^d) |/t^d =0$. On $\{ 0 \in \C_\infty\}$ we have $0 = 0^*$, and for any $x \in \ZZ^d$ either $x\in \C_\infty$, and in this case $T(0,x) = T^*(0,x)$, or $x\notin \C_\infty$, and then $T(0,x ) = +\infty$. Thus on $\{ 0 \in \C_\infty\}$ we have $B_t^v = B_t^{*,v} \cap \C_\infty$. We deduce from Proposition \ref{propcard} that on $\{ 0 \in \C_\infty\}$, a.s., the right hand side of (\ref{eqetape1}) goes to $0$ as $t$ goes to infinity. Fix $\eps>0$. Since f is continuous and $2\B_{\tilde \mu}$ is compact, there exists $t_1 >0$ such that for all $x,y\in 2\B_{\tilde \mu}$, $\|x-y\|_\infty \leq 1/t_1$  implies $|f(x) - f(y)| \leq \eps$. We now approximate $\B_{\tilde \mu}$ by a union of cubes (see Figure \ref{fig2}). For $y\in \RR^d$, we define 
$$ D (y, t) \,=\, \{ y+u \,|\, u \in [-(2t)^{-1},(2t)^{-1}[^d \}$$
the cube of center $y$ and side length $1/t$. We notice that
$$  \bigcup_{y \in \B_{\tilde \mu} \cap \frac{\ZZ^d }{t} } D (y, t) \,=\, \B_{\tilde \mu , t} $$
as defined in (\ref{eqBmut}), thus for all $t \geq t_2 $ large enough we have
$$\L^d \bigg( \B_{\tilde \mu} \triangle \bigg[ \bigcup_{y \in \B_{\tilde \mu} \cap \frac{\ZZ^d }{t} } D (y, t) \bigg] \bigg)  \,= \, \L^d \big( \B_{\tilde \mu} \triangle \B_{\tilde \mu , t}  \big)  \, \leq \, \hat g(\eps) \L^d (\B_{\tilde \mu}) \,,$$
where $\hat g(\eps) = (1+\eps)^d - (1-\eps)^d$ goes to $0$ when $\eps$ goes to $0$. Let $t_0 =  \max (t_1, t_2)$, then we have
\begin{align}
\label{eqetape2}
\bigg| \int_{\B_{\tilde \mu}} f \, d\L^d - \frac{1}{t_0^d} \sum_{y \in \B_{\tilde \mu} \cap \frac{\ZZ^d}{t_0}} f(y) \bigg|& \,\leq\, \|f\|_\infty  \L^d \bigg( \B_{\tilde \mu} \triangle \bigg[ \bigcup_{y \in \B_{\tilde \mu} \cap \frac{\ZZ^d }{t_0} } D (y, t_0) \bigg] \bigg) \nonumber\\
& \qquad\qquad  + \eps \L^d  \bigg( \bigcup_{y \in \B_{\tilde \mu} \cap \frac{\ZZ^d }{t_0} } D (y, t_0)  \bigg) \nonumber \\
& \,\leq \, \|f\|_\infty \hat g(\eps) \L^d(\B_{\tilde \mu}) + \eps (1+\eps)^d \L^d (\B_{\tilde \mu}) \,.
\end{align}
Similarly, for all $t\geq t_0$, 
\begin{align}
\label{eqetape3}
\Bigg| \frac{1}{t^d} \sum_{x\in t\B_{\tilde \mu} \cap \C_\infty} f \left( \frac{x}{t} \right)  & -  \frac{1}{t_0^d} \sum_{y \in \B_{\tilde \mu}\cap \frac{\ZZ^d}{t_0}}  f(y)  \left[ \frac{t_0^d}{t^d} \sum_{x\in tD(y,t_0) \cap \ZZ^d}   \ind{\{ x\in \C_\infty \}} \right] \Bigg| \nonumber\\
&  \,\leq \,  \frac{\|f\|_\infty }{t^d} \bigg| \left( t\B_{\tilde \mu} \cap \ZZ^d  \right) \triangle \bigg( \bigcup_{y\in \B_{\tilde \mu} \cap\frac {\ZZ^d}{t_0}} t D(y,t_0) \cap \ZZ^d \bigg) \bigg| \nonumber \\
& \qquad \qquad + \frac{\eps}{t^d}   \bigg| \bigcup_{y\in \B_{\tilde \mu} \cap\frac {\ZZ^d}{t_0}} t D(y,t_0) \cap \ZZ^d \bigg| \nonumber \\
& \,\leq \, \|f\|_\infty\L^d \left( (1+ \eps) \B_{\tilde \mu} \smallsetminus (1-  \eps) \B_{\tilde \mu} \right) + \eps \L^d\left(  (1+\eps) \B_{\tilde \mu} \right)\nonumber \\
& \, \leq \, \|f\|_\infty \hat g( \eps) \L^{d} (\B_{\tilde \mu}) + \eps (1+\eps)^d \L^d(\B_{\tilde \mu})\,.
\end{align}
We also have
\begin{align}
\label{eqetape4}
\Bigg|  \frac{\theta }{t_0^d} \sum_{y \in \B_{\tilde \mu} \cap \frac{\ZZ^d}{t_0}} f(y) & - \frac{1}{t_0^d} \sum_{y \in \B_{\tilde \mu}\cap \frac{\ZZ^d}{t_0}}  f(y)  \left[ \frac{t_0^d}{t^d} \sum_{x\in tD(y,t_0) \cap \ZZ^d}   \ind{\{ x\in \C_\infty \}} \right] \Bigg| \nonumber\\
& \,\leq \, \|f\|_\infty \frac{\left| \B_{\tilde \mu} \cap \frac{\ZZ^d}{t_0} \right|}{t_0^d} \max_{y \in \B_{\tilde \mu} \cap \frac{\ZZ^d}{t_0}} \left| \left[ \frac{t_0^d}{t^d} \sum_{x\in tD(y,t_0) \cap \ZZ^d}   \ind{\{ x\in \C_\infty \}} \right] - \theta  \right| \nonumber\\
& \,\leq\,\|f\|_\infty \L^d \left( (1+\eps) \B_{\tilde \mu} \right)\max_{y \in \B_{\tilde \mu} \cap \frac{\ZZ^d}{t_0}} \left| \left[ \frac{t_0^d}{t^d} \sum_{x\in tD(y,t_0) \cap \ZZ^d}   \ind{\{ x\in \C_\infty \}} \right] - \theta  \right|  \,.
\end{align}
Since $\B_{\tilde \mu} \cap \frac{\ZZ^d}{t_0^d}$ is finite, it only remains to prove that for all $y \in \B_{\tilde \mu} \cap \frac{\ZZ^d}{t_0^d}$, a.s., we have
\begin{equation}
\label{eqetape5}
\lim_{t\rightarrow \infty}  \frac{t_0^d}{t^d} \sum_{x\in tD(y,t_0) \cap \ZZ^d}   \ind{\{ x\in \C_\infty \}}  \,=\, \theta \,.
\end{equation}
Indeed, Theorem \ref{thmB} is a consequence of (\ref{eqrassemble}), (\ref{eqetape1}), (\ref{eqetape2}), (\ref{eqetape3}), (\ref{eqetape4}) and (\ref{eqetape5}). For all $D$ of the form $D = \prod_{i=1}^d [a_i, b_i[ $ with $0\leq a_i \leq b_i$, we define
$$ X_D \,=\, \frac{1}{\L^d(D)} \sum_{x\in D \cap \ZZ^d} \ind{\{ x\in \C_\infty \}} \,.$$
Then $X$ is a continuous super additive process in the terminology of Ackoglu and Krengel \cite{Ackoglu}. Fix $y \in \B_{\tilde \mu} \cap \ZZ^d / t_0^d$ such that $D(y,t_0) \subset (\RR^+)^d$. Then the family $(tD(y,t_0), t \in \QQ^+)$ is regular (see definition 2.6 in \cite{Ackoglu}), and since $X_{[0,1[^d}$ is bounded (hence integrable), we can apply the ergodic theorem 2.8 in \cite{Ackoglu} to state that
\begin{equation}
\label{eqetape6}
 \lim_{t\rightarrow \infty, t\in \QQ^+} X_{tD(y,t_0)} \,=\, X_\infty \quad \textrm{a.s.}  
 \end{equation}
By ergodicity, $X_\infty$ is constant a.s., and by the dominated convergence theorem, we obtain that $X_\infty = \EE[X_\infty] = \PP[0 \in \C_\infty] = \theta$ a.s. Notice that for $t\notin \QQ^+$, we have $\partial (tD(y,t_0)) \cap \ZZ^d = \emptyset$, thus there exist $t_1, t_2 \in \QQ^+$ such that $t-1<t_1<t<t_2<t+1$ and $X_{t_1D(y,t_0)} =X_{tD(y,t_0)}= X_{t_2D(y,t_0)} $. This implies that the limit in (\ref{eqetape6}) holds along any sequence of $t$ that goes to infinity (not necessarily rational). To prove (\ref{eqetape5}), it is now enough to notice that by symmetry, (\ref{eqetape6}) also holds if $D(y,t_0)$ is included in any other quadrant of $\RR^d$. If $D(y,t_0)$ intersects different quadrants, we can divide it into a finite number of pieces, each one of them intersecting only one quadrant, and treat these different pieces separately.
\end{proof}


\section{Positivity of the time constant}
\label{secpos}

\subsection{Non critical case}
\label{secpos1}

We start the study of the positivity of the time constant with two properties, the first dealing with the case $F(\{ 0 \}) <p_c(d)$, the second with the case $F(\{ 0 \}) >p_c(d)$. The critical case $F(\{ 0 \}) = p_c(d)$ is much more delicate to handle.
\begin{prop}
\label{propF(0)<pc}
Suppose that $F([0,+\infty[)>p_c(d)$. We have
$$ F(\{ 0 \}) < p_c(d) \quad \Longrightarrow \quad  \wmu (e_1) >0 \,. $$
\end{prop}
\begin{proof}
For this proof we can rely on known results concerning finite passage times. Intuitively it is all the more easier to prove that $\wmu (e_1)>0$ when we authorize infinite passage times. For $k \in \RR^{+*}$, we define $t^k(e) = \min (t(e),k)$ and $T^k(x,y)$ (resp. $\wmu^k(e_1)$) is the corresponding minimal passage time (resp. the time constant). Then for all $n\in \NN^*$, we have $T(0,ne_1) \geq T^k(0, ne_1)$. Moreover, choosing an $M\geq k$ in the definition of the regularized times $\wT^k$, we see by Theorem \ref{thmwT} that 
\begin{equation}
\label{eqcv1}
\lim_{n\rightarrow \infty} \frac{T^k (0,ne_1) }{n}\, =\, \lim_{n\rightarrow \infty}  \frac{\wT^k(0,ne_1)}{n} \,=\, \wmu^k (e_1) \,.
\end{equation}
By Theorem \ref{thmT}, we know that
\begin{equation}
\label{eqcv2}
\lim_{n\rightarrow \infty}  \frac{T(0,ne_1)}{n} \,\overset{\L}{=}\, \theta^2 \delta_{\tilde \mu(e_1)} + (1-\theta ^2) \delta_{+\infty}\,.
\end{equation} 
Suppose that $\wmu^k(e_1)>0$. Let $f$ be a continuous real bounded function defined on $\RR \cup \{ +\infty \}$, such that $f\geq 0$, $f(0)=0$ and $f=1$ on $[\wmu^k(e_1)/2, +\infty]$. Then we have
\begin{equation}
\label{eqcv3}
\EE \left[ f \left( \frac{T(0,ne_1)}{n} \right) \right]  \,\geq\, \PP \left[ \frac{T(0,ne_1)}{n} \geq \frac{\wmu^k(e_1)}{2} \right]  \,\geq\, \PP \left[ \frac{T^k (0,ne_1)}{n} \geq \frac{\wmu^k(e_1)}{2} \right]\,.
\end{equation}
By (\ref{eqcv1}), the right hand side of (\ref{eqcv3}) goes to $1$ as $n$ goes to infinity. If $\wmu(e_1)=0$, by (\ref{eqcv2}) we would have $\lim_{n\rightarrow \infty} \EE[f(T(0,ne_1)/n)] = 1-\theta^2 < 1$, thus (\ref{eqcv3}) implies that $\wmu(e_1)>0$. It remains to prove that $\wmu^k(e_1) >0$. Notice that the times $(t^k(e), e \in \EE^d)$ are finite, and if we denote by $F^k$ their common law, we have $F^k(\{0\}) = F(\{0\}) <p_c(d)$. Proposition 5.8 in \cite{Kesten:StFlour} states that there exist constants $0<C,D,E<\infty$, depending on $d$ and $F^k$, such that
\begin{equation}
\label{eqcv4}
\PP \left[ \textrm{there exists a self avoiding path } \gamma \textrm{ starting at }0 \textrm{ s.t. } |\gamma| \geq n \textrm{ and } T^k(\gamma) <Cn \right] \,\leq\, De^{-En}\,.
\end{equation}
Equation (\ref{eqcv4}) implies that $\PP[T^k(0,ne_1) < Cn] \leq De^{-En}$, thus $\wmu(e_1) \geq C >0$.
\end{proof}

\begin{prop}
\label{propF(0)>pc}
Suppose that $F([0,+\infty[)>p_c(d)$. We have
$$ F(\{ 0 \}) > p_c(d) \quad \Longrightarrow \quad  \wmu (e_1) =0 \,. $$
\end{prop}

\begin{proof}
Suppose that $F(\{ 0 \}) > p_c(d)$. We denote by $\C_0$ the a.s. unique infinite cluster of the percolation $(\ind{\{ t(e)=0 \}}, e\in \EE^d)$, that is super-critical. By Theorem \ref{thmT}, we know that 
\begin{equation}
\label{eqcv5}
\lim_{n\rightarrow \infty}  \frac{T(0,ne_1)}{n} \,=\, \theta^2 \delta_{\tilde \mu(e_1)} + (1-\theta ^2) \delta_{+\infty}\,.
\end{equation} 
Suppose that $\wmu (e_1)>0$. Then there exists a continuous real bounded function $f$ defined on $\RR \cup \{ +\infty \}$ such that $f\geq 0$, $f(\wmu(e_1)) = f(+\infty) =0$ and $f(0) >0$. Applying the FKG inequality we have
$$ \EE \left[ f\left( \frac{T(0,ne_1)}{n} \right) \right]  \,\geq\, f(0) \,\PP[ T(0,ne_1)=0] \,\geq \, f(0)\, \PP [0\in \C_0 \,,\, ne_1 \in \C_0] \,\geq\, f(0)\, \PP[0\in \C_0]^2 \,>\, 0\,.$$
By (\ref{eqcv5}), we also know that $\lim_{n\rightarrow \infty} \EE[ f(T(0,ne_1)/n)] =0$, which is a contradiction, thus $\wmu(e_1) =0$.
\end{proof}


\subsection{Lower large deviations}
\label{secLD}

Our goal is now to handle the critical case. We define the hyperplane $H_{n} = \{ z \in \RR^d \,|\, z_1=n  \}$, and the time 
$$ T(0,H_n) \,=\, \inf \{ T(0,x) \,|\, x\in \ZZ^d \cap H_n  \}\,. $$
This is the so called {\em point-to-line} passage time (we should say point-to-hyperplane, but the term point-to-line has been kept from dimension $2$). The proof of Theorem \ref{thmpositivity} relies on the following property.
\begin{prop}
\label{propGD}
Suppose that there exists $M\in ]0,+\infty[$ such that $F([0,+\infty[) = F([0,M])>p_c(d)$. If $\wmu(e_1)>0$, then for all $\eps >0 $, there exist positive constants $C_5,C_6$ such that
$$ \PP[0\in \C_M \,,\, T(0,H_n) < (\wmu(e_1) -\eps) n ]  \,\leq\, C_5 e^{-C_6 n} \,.$$
\end{prop}

\begin{rem}
In proposition \ref{propGD}, we restrict our study to laws $F$ such that $F(]M,+\infty [) =0$. In this case, the percolations $(\ind{\{ t(e) \leq M \}}, e\in \EE^d)$ and $(\ind{\{ t(e) <\infty \}}, e\in \EE^d)$ are the same, and $\C_M = \C_\infty$. This property makes our study more tractable.
\end{rem}

Before proving Proposition \ref{propGD}, we need two preliminary results. We define the regularized point-to-line passage time 
$$\wT (0, H_n) \,=\, \inf \{ T(\w0, \wx) \,|\, x \in H_n \} \,. $$
\begin{prop}
\label{propbn}
Suppose that $F([0,+\infty[) > p_c(d)$. Then
$$ \lim_{n\rightarrow \infty} \frac{\wT(0,H_n)}{n} \,=\, \wmu (e_1) \quad \textrm{a.s.}  $$
\end{prop}

\begin{proof}
We have $\wT(0, H_n) \leq \wT (0, ne_1)$ thus
$$ \limsup_{n\rightarrow \infty} \frac{\wT(0,H_n)}{n} \,\leq \lim_{n\rightarrow \infty} \frac{\wT(0, ne_1)}{n} \,=\, \wmu(e_1) \quad \textrm{a.s.} $$
If $\wmu(e_1) =0$, Proposition \ref{propbn} is proved. Suppose that $\wmu(e_1) > 0$. If 
$$\liminf_{n\rightarrow \infty} \frac{\wT(0,H_n)}{n}\, <\, \wmu(e_1)\,,$$
it implies that there exist $\eps>0$, a sequence $n_k \in \NN^*$ such $\lim_{k\rightarrow \infty} n_k = +\infty$ and a sequence of vertices $y_k \in H_{n_k}$ such that
\begin{equation}
\label{eqbn1}
 \lim_{k\rightarrow \infty} \frac{\wT (0,y_k)}{n_k} \,\leq \, \wmu(e_1) - \eps \,. 
\end{equation}
Since $\|y_k\| \geq n_k$, by Lemma \ref{lem2} we know that a.s., for all $k$ large enough,
\begin{equation}
\label{eqbn2}
\frac{\wT(0,y_k)}{\|y_k\|_1} \,\geq \, \wmu \left( \frac{y_k}{\|y_k\|_1} \right) - \frac{\eps}{2} \,.
\end{equation}
By (\ref{eqmaj}) we know that whatever $y \in \ZZ^d$, $\wmu ( y/\|y\|_1 ) \geq \wmu (e_1)$, thus inequalities (\ref{eqbn1}) and (\ref{eqbn2}) are in contradiction. This ends the proof of Proposition \ref{propbn}.
\end{proof}

In the proof of proposition \ref{propGD}, we would like to make appear of sequence of independent point-to-line passage times. However, there would be too many possible starting points for these point-to-line passage times. To control the combinatorial factor that will appear, we use instead some "plaquette-to-line" passage times. The following lemma helps us to control the difference between a path starting at a point and a path starting somewhere in a small plaquette near this point.
\begin{lem}
\label{lemDM}
Suppose that $F([0,+\infty[) > p_c(d)$. For any $\alpha >0$, there exists $\delta_0 >0$ such that for all $\delta \leq \delta_0$, we have
$$ \lim_{K \rightarrow \infty} \PP \left[ \max \left\{ D_\infty (\w0, \wy) \,|\, y \in \ZZ^d \cap \left( \{0\} \times [0, \delta K[^{d-1}\right) \right\} >\alpha K \right] \,=\, 0 \,. $$
\end{lem}

\begin{proof}
For every edge $e\in \EE^d$, we define the time $t'(e)$  by
$$ t'(e) \,=\, \left\{ \begin{array}{ll} 1 & \textrm{ if }t(e)<\infty \,, \\ +\infty & \textrm{ if }t(e)=+\infty \,. \end{array} \right.  $$
The law of $t'(e)$ is $F' = F([0,\infty[) \delta_{1} + F(\{ +\infty \}) \delta_{+\infty}$. We denote by $T', \wT', \wmu', \wB_t', \B_{\tilde \mu'}$ the objects associated with these new times $(t'(e), e \in \EE^d)$. Notice that $D_\infty (\w0, \wy) = \wT'(0,y)$, thus we have
\begin{align}
\label{eqDM1}
& \left\{ \max \left\{ D_\infty (\w0, \wy) \,|\, y \in \ZZ^d \cap \left(\{0\} \times [0, \delta K[^{d-1}\right) \right\} > \alpha K  \right\} \nonumber\\
&\qquad\qquad \qquad\,= \,\left\{ \max \left\{ \wT'(0,y) \,|\, y \in \ZZ^d \cap \left(\{0\} \times [0, \delta K[^{d-1}\right) \right\} > \alpha K  \right\} \nonumber \\
&\qquad\qquad \qquad \,= \, \{ \{0\} \times [0, \delta K[^{d-1} \not\subset \wB'_{\alpha K} \}  \nonumber \\
&\qquad\qquad \qquad \,= \, \left\{ \{0\} \times [0, \delta /\alpha [^{d-1} \not\subset \frac{\wB'_{\alpha K}}{\alpha K} \right\} \,.
\end{align}
Since $F' ([0,+\infty[) = F([0,+\infty[) > p_c(d)$ and $F'(\{0\}) = 0 < p_c(d)$, we know by Proposition \ref{propF(0)<pc} that $\wmu'(e_1) >0$, thus there exists $\delta_0 >0$ such that for all $\delta \leq \delta_0$, 
\begin{equation}
\label{eqDM2}
 \{0\} \times [0, \delta /\alpha [^{d-1} \,\subset\, \frac{1}{2} \B_{\tilde \mu'} \,.
\end{equation}
Combining (\ref{eqDM1}) and (\ref{eqDM2}), we obtain for all $\delta \leq \delta_0$
$$\PP  \left[ \max \left\{ D_\infty (\w0, \wy) \,|\, y \in \ZZ^d \cap \left( \{0\} \times [0, \delta K[^{d-1}\right) \right\} >\alpha K \right]  \,\leq \, \PP \left[  \frac{1}{2} \B_{\tilde \mu'} \not\subset \frac{\wB'_{\alpha K}}{\alpha K} \right] \,.$$
Moreover, since $\wmu'(e_1) >0$, the strong shape theorem applies to the sets $B'_t$, thus
$$ \lim_{K \rightarrow \infty}\PP \left[ \frac{1}{2} \B_{\tilde \mu'} \not\subset \frac{\wB'_{\alpha K}}{\alpha K} \right]  \,=\, 0 \, .$$
\end{proof}

\begin{proof}[Proof of Proposition \ref{propGD}]
We follow Kesten's proof of Inequality (5.3) of Theorem (5.2) in \cite{Kesten:StFlour}, which relies on the proof of Proposition (5.23). We do minor adaptations, but for the sake of completeness, we present the entire proof. In what follows, the $M$ we choose to define the regularized times $\wT$ satisfies $F(]M,+\infty[) =0$, thus $\C_\infty = \C_M$. Fix $n\in \NN^*$. On the event $\{ 0\in \C_M \,,\, T(0,H_n) < x \}$, there exists a self avoiding path $\gamma = (v_0 =0, e_1, v_1, \dots ,e_p, v_p)$ from $0$ to a point $v_p \in H_n$ such that $T(\gamma) <x$. We have necessarily $p=|\gamma| \geq n$ and $\gamma \subset \C_\infty $. We want to cut this path $\gamma$ into several pieces. Fix $N,K\in \NN^*$ to be chosen later, with $N\leq K$. Let $\tau (0)= 0$ and $a_0=v_{\tau (0)} =0$. For all $i\geq 0$, let 
$$\tau(i+1) \, =\, \min \{ k> \tau( i) \,|\, \| a_i - v_k \|_\infty = K+N  \}$$
and $a_{i+1} = v_{\tau (i+1)}$ as long as the set $\{ k> \tau( i) \,|\, \| a_i - v_k \|_\infty = K+N  \}$ is non empty (see Figure \ref{fig3}).
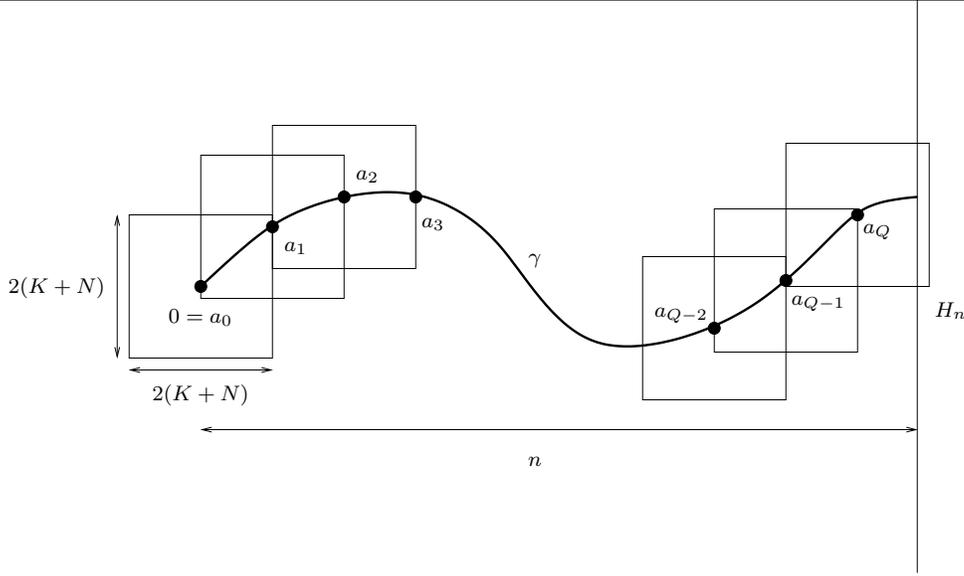
\begin{figure}[h!]
\centering
\begin{picture}(0,0)%
\includegraphics{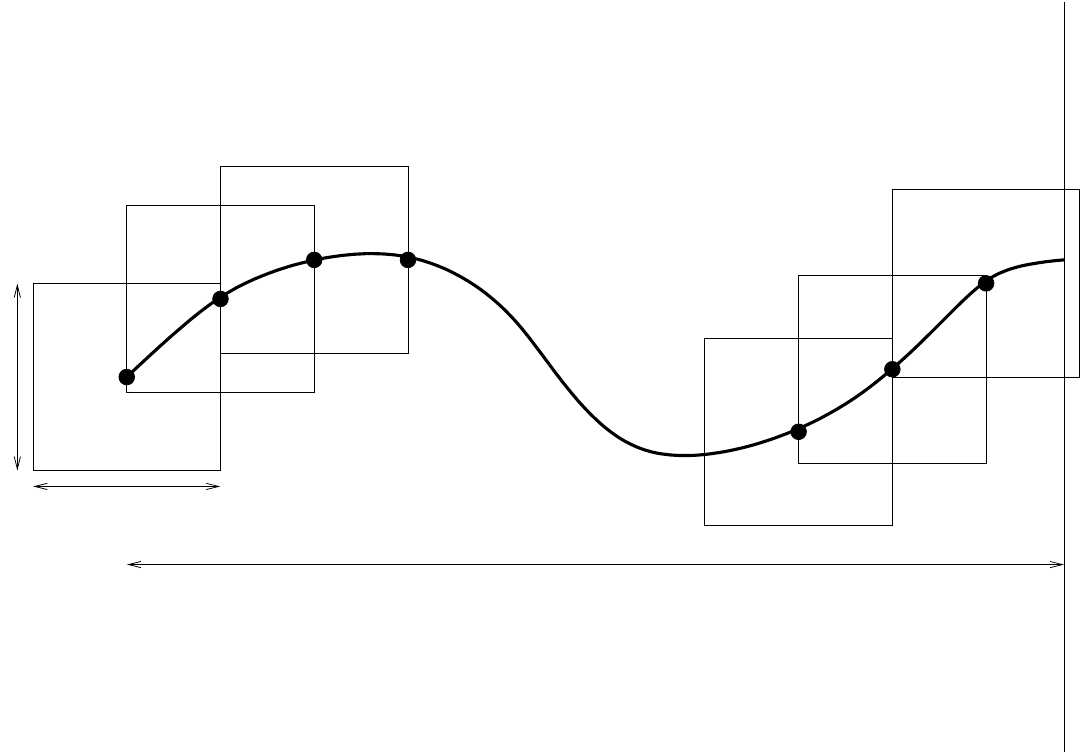}%
\end{picture}%
\setlength{\unitlength}{1973sp}%
\begingroup\makeatletter\ifx\SetFigFont\undefined%
\gdef\SetFigFont#1#2#3#4#5{%
  \reset@font\fontsize{#1}{#2pt}%
  \fontfamily{#3}\fontseries{#4}\fontshape{#5}%
  \selectfont}%
\fi\endgroup%
\begin{picture}(10455,7224)(1486,-7273)
\put(10126,-3886){\makebox(0,0)[lb]{\smash{{\SetFigFont{8}{9.6}{\rmdefault}{\mddefault}{\updefault}{\color[rgb]{0,0,0}$a_{Q-1}$}%
}}}}
\put(11926,-4036){\makebox(0,0)[lb]{\smash{{\SetFigFont{8}{9.6}{\rmdefault}{\mddefault}{\updefault}{\color[rgb]{0,0,0}$H_n$}%
}}}}
\put(4651,-2311){\makebox(0,0)[lb]{\smash{{\SetFigFont{8}{9.6}{\rmdefault}{\mddefault}{\updefault}{\color[rgb]{0,0,0}$a_2$}%
}}}}
\put(2701,-5086){\makebox(0,0)[b]{\smash{{\SetFigFont{8}{9.6}{\rmdefault}{\mddefault}{\updefault}{\color[rgb]{0,0,0}$2(K+N)$}%
}}}}
\put(6901,-5911){\makebox(0,0)[b]{\smash{{\SetFigFont{8}{9.6}{\rmdefault}{\mddefault}{\updefault}{\color[rgb]{0,0,0}$n$}%
}}}}
\put(1501,-3736){\makebox(0,0)[rb]{\smash{{\SetFigFont{8}{9.6}{\rmdefault}{\mddefault}{\updefault}{\color[rgb]{0,0,0}$2(K+N)$}%
}}}}
\put(2701,-4111){\makebox(0,0)[b]{\smash{{\SetFigFont{8}{9.6}{\rmdefault}{\mddefault}{\updefault}{\color[rgb]{0,0,0}$0=a_0$}%
}}}}
\put(9076,-4036){\makebox(0,0)[rb]{\smash{{\SetFigFont{8}{9.6}{\rmdefault}{\mddefault}{\updefault}{\color[rgb]{0,0,0}$a_{Q-2}$}%
}}}}
\put(6901,-3361){\makebox(0,0)[b]{\smash{{\SetFigFont{8}{9.6}{\rmdefault}{\mddefault}{\updefault}{\color[rgb]{0,0,0}$\gamma$}%
}}}}
\put(3751,-3211){\makebox(0,0)[lb]{\smash{{\SetFigFont{8}{9.6}{\rmdefault}{\mddefault}{\updefault}{\color[rgb]{0,0,0}$a_1$}%
}}}}
\put(5476,-2911){\makebox(0,0)[lb]{\smash{{\SetFigFont{8}{9.6}{\rmdefault}{\mddefault}{\updefault}{\color[rgb]{0,0,0}$a_3$}%
}}}}
\put(11026,-2986){\makebox(0,0)[lb]{\smash{{\SetFigFont{8}{9.6}{\rmdefault}{\mddefault}{\updefault}{\color[rgb]{0,0,0}$a_Q$}%
}}}}
\end{picture}%
\caption{Construction of the points $a_i$.}
\label{fig3}
\end{figure}
When this set becomes empty, we define $Q=i$ and we stop the process. For $z\in \ZZ^d$, we denote by $z(i)$ the $i$-th coordinate of $z$. By definition, we know that $a_{i} (1) \leq a_{i-1} (1) + K+N$, for all $i\leq Q$. Since $v_p \in H_n$, we know that
\begin{equation}
\label{eqQ}
n \,=\, v_p(1) \,<\, a_Q(1) + K+N \,\leq\, (Q+1)(K+N) \,.
\end{equation}
By definition of the points $a_i$, for all $i\in\{0,\dots , Q-1\}$, there exists $\nu(i)\in \{1,\dots , d\}$ and $\eta (i) \in \{ -1, +1 \}$ such that 
$$ a_{i+1} (\nu(i)) \,= \, a_i (\nu(i)) + \eta(i) (K+N) \,.$$
We want to compare the part of $\gamma$ between $a_i$ and $a_{i+1}$ with a point-to-line passage time. The variables $\nu(i)$ and $\eta(i)$ tell us the direction and the sense in which this point-to-line displacement happens. However, the starting point $a_i$ is random, and may be located at too many different positions. Thus we define a shorter piece of $\gamma$ that we can describe with less parameters. We follow the notations of Kesten's proof of Proposition (5.23) in \cite{Kesten:StFlour} (see Figure \ref{fig4}).
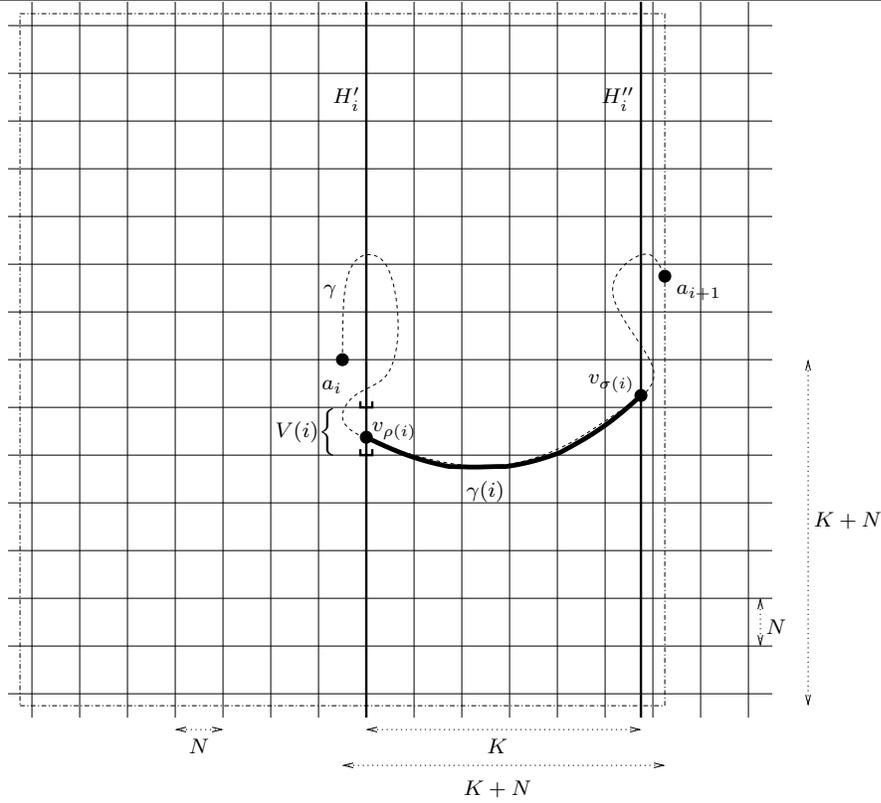
\begin{figure}[h!]
\centering
\begin{picture}(0,0)%
\includegraphics{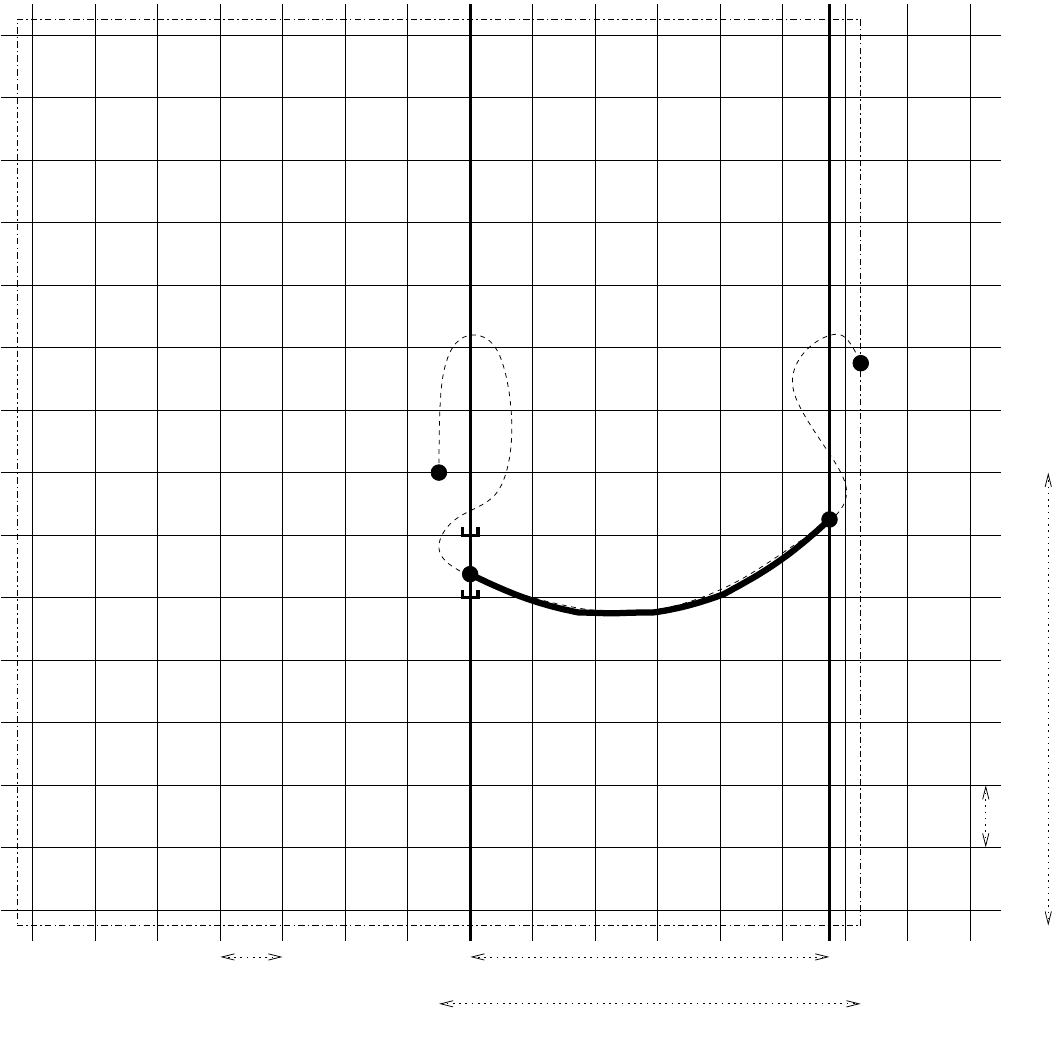}%
\end{picture}%
\setlength{\unitlength}{1973sp}%
\begingroup\makeatletter\ifx\SetFigFont\undefined%
\gdef\SetFigFont#1#2#3#4#5{%
  \reset@font\fontsize{#1}{#2pt}%
  \fontfamily{#3}\fontseries{#4}\fontshape{#5}%
  \selectfont}%
\fi\endgroup%
\begin{picture}(10152,10095)(289,-8623)
\put(4426,-4036){\makebox(0,0)[rb]{\smash{{\SetFigFont{8}{9.6}{\rmdefault}{\mddefault}{\updefault}{\color[rgb]{0,0,0}$V(i) \bigg\{$}%
}}}}
\put(4501,-3436){\makebox(0,0)[rb]{\smash{{\SetFigFont{8}{9.6}{\rmdefault}{\mddefault}{\updefault}{\color[rgb]{0,0,0}$a_i$}%
}}}}
\put(8701,-2236){\makebox(0,0)[lb]{\smash{{\SetFigFont{8}{9.6}{\rmdefault}{\mddefault}{\updefault}{\color[rgb]{0,0,0}$a_{i+1}$}%
}}}}
\put(4876,-3961){\makebox(0,0)[lb]{\smash{{\SetFigFont{8}{9.6}{\rmdefault}{\mddefault}{\updefault}{\color[rgb]{0,0,0}$v_{\rho(i)}$}%
}}}}
\put(6301,-4786){\makebox(0,0)[b]{\smash{{\SetFigFont{8}{9.6}{\rmdefault}{\mddefault}{\updefault}{\color[rgb]{0,0,0}$\gamma(i)$}%
}}}}
\put(8176,-3361){\makebox(0,0)[rb]{\smash{{\SetFigFont{8}{9.6}{\rmdefault}{\mddefault}{\updefault}{\color[rgb]{0,0,0}$v_{\sigma(i)}$}%
}}}}
\put(8176,164){\makebox(0,0)[rb]{\smash{{\SetFigFont{8}{9.6}{\rmdefault}{\mddefault}{\updefault}{\color[rgb]{0,0,0}$H''_i$}%
}}}}
\put(4726,164){\makebox(0,0)[rb]{\smash{{\SetFigFont{8}{9.6}{\rmdefault}{\mddefault}{\updefault}{\color[rgb]{0,0,0}$H'_i$}%
}}}}
\put(4426,-2236){\makebox(0,0)[rb]{\smash{{\SetFigFont{8}{9.6}{\rmdefault}{\mddefault}{\updefault}{\color[rgb]{0,0,0}$\gamma$}%
}}}}
\put(6451,-8011){\makebox(0,0)[b]{\smash{{\SetFigFont{8}{9.6}{\rmdefault}{\mddefault}{\updefault}{\color[rgb]{0,0,0}$K$}%
}}}}
\put(6451,-8536){\makebox(0,0)[b]{\smash{{\SetFigFont{8}{9.6}{\rmdefault}{\mddefault}{\updefault}{\color[rgb]{0,0,0}$K+N$}%
}}}}
\put(10426,-5161){\makebox(0,0)[lb]{\smash{{\SetFigFont{8}{9.6}{\rmdefault}{\mddefault}{\updefault}{\color[rgb]{0,0,0}$K+N$}%
}}}}
\put(9826,-6511){\makebox(0,0)[lb]{\smash{{\SetFigFont{8}{9.6}{\rmdefault}{\mddefault}{\updefault}{\color[rgb]{0,0,0}$N$}%
}}}}
\put(2701,-8011){\makebox(0,0)[b]{\smash{{\SetFigFont{8}{9.6}{\rmdefault}{\mddefault}{\updefault}{\color[rgb]{0,0,0}$N$}%
}}}}
\end{picture}%

\caption{Construction of the path $\gamma(i)$.}
\label{fig4}
\end{figure}
For $i \in \{ 0,\dots , Q-1 \}$, we define
$$  \rho (i) \,=\, \max \left\{ k \in \{ \tau(i) , \dots , \tau(i+1)  \} \,|\, v_k (\nu(i)) = \left( \left\lfloor \frac{a_i (\nu(i))}{N} \right\rfloor +\frac{1+\eta(i)}{2} \right) N \right\} $$
and
$$  \sigma (i) \,=\, \min \left\{ k \in \{ \rho(i)  , \dots , \tau(i+1)  \} \,|\, v_k (\nu(i)) = \left( \left\lfloor \frac{a_i (\nu(i))}{N} \right\rfloor +\frac{1+\eta(i)}{2} \right) N +  \eta(i) K \right\} \,.$$
We have 
$$  \tau (i) \,\leq \,\rho (i)\, <\, \sigma (i) \,\leq\, \tau(i+1) \,$$
and we denote by $\gamma(i)$ the piece of $\gamma$ between $v_{\rho(i)}$ and $v_{\sigma (i)}$. By construction, $\gamma(i)$ lies strictly between the hyperplanes
$$ H_i' \,=\,  \left\{ x\in \RR^d \,|\, x  (\nu(i)) = \left( \left\lfloor \frac{a_i (\nu(i))}{N} \right\rfloor +\frac{1+\eta(i)}{2} \right) N \right\}$$
and 
$$ H_i'' \,=\,  \left\{ x\in \RR^d \,|\, x  (\nu(i)) = \left( \left\lfloor \frac{a_i (\nu(i))}{N} \right\rfloor +\frac{1+\eta(i)}{2} \right) N+  \eta(i) K \right\}\,,$$
except for its first and last points which belong to these hyperplanes. The plaquettes of the form
$$ V(i, \lambda) \,=\, \{ x\in H'(i) \,|\, \forall j \neq \nu(i)\,,\, x(j) \in [\lambda (j) N, (\lambda (j)+1) N [  \} \,,\quad \lambda(j) \in \ZZ \,, 1\leq j \leq d\,,\, j\neq \nu(i) $$
form a tiling of $H'(i)$. We denote by $\lambda_i = (\lambda_i(j), j\neq \nu(i))$ the unique choice of $\lambda$ such that $V(i,\lambda)$ contains $v_{\rho(i)}$, and we define $V(i) = V(i, \lambda_i)$. Remember that for all $k\in \{ \tau(i), \dots , \tau(i+1) \}$, $\|a_i - v_k \|_\infty \leq K+N$. Thus by construction the pieces $\gamma(i)$, $i\in \{0,\dots , Q-1\}$, have the following properties :
\begin{itemize}
\item for $i\neq j$, the paths $\gamma(i)$ and $\gamma(j)$ are edge disjoint ;
\item $\gamma(i) \subset \C_\infty$ ;
\item $\gamma(i)$ connects $V(i)$ to $H''(i)$ ;
\item $\gamma(i)$ is included in the box $G(i)$ defined by
$$ G(i) \, = \, \left\{ x\in \RR^d \,\bigg|\, \begin{array}{c} \forall j \neq \nu(i) \,, \, x(j) \in [ \lambda_i(j) -2 (K+N) , \lambda_i(j) + 2(K+N) ]  \textrm{ and }\\ x(\nu(i)) \in \left[ \left( \left\lfloor \frac{a_i (\nu(i))}{N} \right\rfloor +\frac{1+\eta(i)}{2} \right) N , \left( \left\lfloor \frac{a_i (\nu(i))}{N} \right\rfloor +\frac{1+\eta(i)}{2} \right) N + \eta(i) K \right] \end{array} \right\} \,.$$
\end{itemize}
The path $\gamma(i)$ is a "plaquette-to-face" path in the box $G(i)$. Notice that $G(i)$ is completely determined by the triplet $\Lambda (i) := (\nu (i), \eta (i), V(i))$. We obtain that
\begin{align}
\label{eqev}
 \PP & [0 \in \C_M \,,\, T(0,H_n) < x] \nonumber\\
 & \quad \,\leq\, \sum_{Q \geq \frac{n}{K+N} -1} \,\, \sum_{\tiny{\begin{array}{c}\Lambda (i),\\ i=0,\dots , Q-1 \end{array}} } \PP \left[ \begin{array}{c} \exists \gamma \textrm{ self avoiding path s.t. } \cup_{i=0}^{Q-1} \gamma(i) \subset \gamma \subset \C_\infty \,,\\ \sum_{i=0}^{Q-1} T(\gamma(i)) <x \textrm{ and }\\ \forall i \,,\, \gamma(i) \textrm{ crosses } G(i) \textrm{ from } V(i) \textrm{ to } H''_i  \end{array} \right] \,.
 \end{align}
Let us define the event
$$ A(\Lambda (i), x_i) \,=\, \left\{ \begin{array}{c}   \exists \gamma(i) \textrm{ self avoiding path s.t. }  \gamma(i) \subset \C_\infty \,,\, T(\gamma(i)) <x_i \\ \textrm{and } \gamma(i) \textrm{ crosses } G(i) \textrm{ from } V(i)  \textrm{ to } H''_i  \end{array}  \right\} \,.$$
If $A$ and $B$ are two events, we denote by $A\circ B$ the event that $A$ and $B$ occur disjointly (see (4.17) in \cite{Kesten:StFlour} for a precise definition). Inequality (\ref{eqev}) can be reformulated as
\begin{align}
\label{eqev2}
 \PP &  [0 \in \C_M \,,\, T(0,H_n) < x] \nonumber\\
 & \,\leq\, \sum_{Q \geq \frac{n}{K+N} -1} \,\, \sum_{\tiny{\begin{array}{c}\Lambda (i),\\ i=0,\dots , Q-1 \end{array}} } \PP \left[ \bigcup_{\tiny{\begin{array}{c} x_i \in \QQ^+ \,,\, i=0,\dots , Q-1 \\ \textrm{s.t. } \sum_{i=0}^{Q-1} x_i <x \end{array} }}  A(\Lambda (0), x_0) \circ  \cdots \circ A(\Lambda (Q-1), x_{Q-1}) \right] \,.
 \end{align}
The events $A(\Lambda (i), x_i)$, $i \in \{0, \dots , Q-1  \}$ are decreasing (notice that since $t'(e) \leq t(e)$ for all $e$, then the infinite cluster of edges of finite time $t(e)$ is included in the infinite cluster of edges of finite time $t'(e)$). Thus we can apply Theorem (4.8) in \cite{Kesten:StFlour}. Let $(A'(\Lambda (i), x_i))$ be events defined on a new probability space $(\Omega', \PP')$ such that the families $(A'(\Lambda (i), x_i), x_i \in \QQ^+)$ for $i\in \{0,\dots , Q-1\}$ are independent, and for each $i$ the joint distribution of $(A'(\Lambda (i), x_i), x_i \in \QQ^+)$ under $\PP'$ is the same as the joint distribution of $(A(\Lambda (i), x_i), x_i \in \QQ^+)$ under $\PP$. Then
\begin{align}
\label{eqev3}
\PP & \left[ \bigcup_{\tiny{\begin{array}{c} x_i \in \QQ^+ \,,\, i=0,\dots , Q-1 \\ \textrm{s.t. }\sum_{i=0}^{Q-1} x_i <x \end{array} }}  A(\Lambda (0), x_0) \circ  \cdots \circ A(\Lambda (Q-1), x_{Q-1}) \right] \nonumber\\
&\qquad \qquad \qquad  \,\leq\, \PP' \left[ \bigcup_{\tiny{\begin{array}{c} x_i \in \QQ^+ \,,\, i=0,\dots , Q-1 \\ \textrm{s.t. }\sum_{i=0}^{Q-1} x_i <x \end{array} }}  \bigcap_{i=0}^{Q-1} A'(\Lambda (i), x_i)  \right]\,.
\end{align}
Let $(Y_i(K,N), i\in\{0,\dots ,Q-1 \})$ be i.i.d. random variables with the same law as $Y(K,N)$ defined by 
$$ Y (K,N) \,=\, \inf \{ T(\gamma) \,|\,  \gamma \textrm{ is a path from } \{0\} \times [0,N[^{d-1} \textrm{ to } H_K \textrm{ and } \gamma \subset \C_{\infty} \} \,,$$
with $Y(M,N)=+\infty$ if such a path does not exists. Then combining (\ref{eqev2}) and (\ref{eqev3}), and relaxing the constraints required on the paths $\gamma(i)$, we obtain
\begin{align}
\label{eqev4}
 \PP  [0 \in \C_M \,,\, T(0,H_n) < x] & \,\leq\, \sum_{Q \geq \frac{n}{K+N} -1} \,\, \sum_{\tiny{\begin{array}{c}\Lambda (i),\\ i=0,\dots , Q-1 \end{array}} }  \PP \left[ \sum_{i=0}^{Q-1} Y_i(K,N) <x \right] \nonumber \\
 & \,\leq\, \sum_{Q \geq \frac{n}{K+N} -1}  \left[ 2d \left( 15 K/N \right)^d \right]^Q  \PP \left[ \sum_{i=0}^{Q-1} Y_i(K,N) <x \right]  \,.
\end{align}
In (\ref{eqev4}), the combinatorial estimate has been derived by Kesten's proof of Proposition (5.23) in \cite{Kesten:StFlour}. Inequality (\ref{eqev4}) is a complete analog of Kesten's Proposition (5.23), except that we require in the definition of $Y(K,N)$ that the path $\gamma$ lies in $\C_\infty$. This allows us to compare $Y(K,N)$ with the time $\wT (0,H_K) = \inf \{ T(\w0, \wx ) \,|\, x \in H_K \}$. We make a crucial use of the fact that $\C_M = \C_\infty$ to obtain that if $\gamma $ is a path as in the definition of $Y(K,N)$, and if we denote by $y$ its starting point in $ \{0\} \times [0,N[^{d-1}  \cap \, \C_\infty$, then
$$ \wT(0,H_K) \,\leq \, T(\gamma) + T(\w0, y) \,\leq\, T(\gamma) + M D_M (\w0 , y)\,.$$
Taking the infimum over such paths $\gamma$, we obtain that
\begin{align}
\label{eqev5}
\wT (0,H_K) \,\leq\, Y(K,N) + M \max \left\{ D_M (\w0, y) \,|\, y\in \left(\{0\} \times [0,N[^{d-1} \right) \cap \,\C_\infty  \right\}\,.
\end{align}
This will help us to control the right hand side of (\ref{eqev4}). Take $x=n(\wmu (e_1) - \eps)$ in (\ref{eqev4}). Following Kesten's proof of Proposition (5.29) in \cite{Kesten:StFlour}, we have for all $\lambda>0$ the following upper bounds :
\begin{align}
\label{eqev6}
 \PP &  \left[ \sum_{i=0}^{Q-1} Y_i(K,N) < n(\wmu(e_1) - \eps) \right] \nonumber \\
 & \,\leq \, e^{\lambda n (\tilde \mu(e_1) -\eps)}  \, \EE \left[ e^{-\lambda Y(K,N)} \right]^Q \nonumber  \\
& \,\leq \, e^{\lambda (K+N)\tilde \mu(e_1)}\,\left[ \EE\left[  e^{-\lambda Y(K,N)}\right] e^{\lambda (K+N) (\tilde \mu(e_1) - \eps)}  \right]^Q \nonumber  \\
& \,\leq \,  e^{\lambda (K+N)\tilde \mu(e_1)}\, \left[  \left( e^{-\lambda K (\tilde \mu(e_1) - \eps/2)} + \PP \left[ Y(K,N) < K (\wmu(e_1) - \eps/2) \right] \right) e^{\lambda (K+N) (\tilde \mu(e_1) - \eps)}  \right]^Q \nonumber  \\
& \,\leq \, e^{\lambda (K+N)\tilde \mu(e_1)}\,\nonumber  \\
& \quad \times \left[ e^{\lambda \left[- K (\tilde \mu(e_1) - \eps/2) + (K+N) (\tilde \mu(e_1) - \eps) \right]} + e^{\lambda (K+N) (\tilde \mu(e_1) - \eps)} \, \PP \left[ Y(K,N) < K (\wmu(e_1) - \eps/2) \right]  \right]^Q \,.
\end{align}
Take $N = \delta K$, with $\delta \leq 1$. There exists  $ \delta_1 (\wmu(e_1), \eps) >0$ such that for all $\delta \leq \delta_1$, and for $\lambda = R/K$, we obtain
\begin{align}
\label{eqev7}
e^{\lambda \left[- K (\tilde \mu(e_1) - \eps/2) + (K+N) (\tilde \mu(e_1) - \eps) \right]} & \,\leq \, e^{- R \left[ \tilde \mu(e_1) - \eps/2 - (1+\delta) (\tilde \mu(e_1)-\eps) \right]} \nonumber \\
& \,\leq\, e^{-R (\eps/2 - \delta (\tilde \mu(e_1) - \eps))} \,\leq \, e^{-R\eps/4} \,,
\end{align}
and we can choose $R$ large enough (how large depending only on $\eps$) so that the right hand side of (\ref{eqev7}) is as small as we want. Moreover, using (\ref{eqev5}) we have
\begin{align}
\label{eqev8}
e^{\lambda (K+N) (\tilde \mu(e_1) - \eps)}  & \, \PP \left[ Y(K,N) < K (\wmu(e_1) - \eps/2) \right] \nonumber \\
& \,\leq \, e^{2 R \tilde \mu(e_1)} \, \max_{N\leq \delta K} \PP \left[ Y(K,N) < K (\wmu(e_1) - \eps/2) \right] \nonumber \\
& \,\leq \, e^{2 R \tilde \mu(e_1)} \bigg(  \PP\left[ \wT (0,H_K)  <  K (\wmu(e_1) - \eps/4) \right] \nonumber\\
&\qquad  \qquad+ \PP \left[ \max \left\{ D_M (\w0, \wy) \,|\, y\in \ZZ^d \cap \left( \{0\} \times [0,\delta K[^{d-1} \right) \right\} \geq \frac{\eps K}{4M}\right] \bigg) \,.
\end{align}
According to Proposition \ref{propbn} and Lemma \ref{lemDM}, there exists $\delta_0 (M,\eps)$ such that for all $\delta\leq \delta_0$, for any fixed $R$, we can choose $K$ large enough so that the right hand side of (\ref{eqev8}) is as small as we want. Combining (\ref{eqev6}), (\ref{eqev7}) and (\ref{eqev8}), we see that we can choose our parameters $\delta \leq \min (\delta_0, \delta_1)$ and $R$ such that for $K$ large enough, the probability appearing in the right hand side of (\ref{eqev4}) compensates the combinatorial term, and Proposition \ref{propGD} is proved. 
\end{proof}


\subsection{Critical case}
\label{secpos2}

We can now study the case $F(\{0\}) = p_c(d)$. We start with specific laws $F$.
\begin{prop}
\label{propF(0)=pc}
Let $K \in ]0,+\infty[$ and $\eta >0$. Let $F = p_c(d) \delta_0 + \eta \delta_{K} + (1-p_c(d) - \eta) \delta_{\infty}$. Then the corresponding time constant is null : $\wmu (e_1) =0$.
\end{prop}

\begin{proof}
We follow Kesten's proof of the nullity of the time constant in the case of finite passage times (see Theorem 6.1 in \cite{Kesten:StFlour}). Let $C_0(0)$ (resp. $C_M(0)$) be the open cluster of $0$ in the percolation $(\ind{\{ t(e)=0 \}}, e\in \EE^d)$ (resp. $(\ind{\{ t(e)\leq M \}}, e\in \EE^d)$). Notice that, since $F([0,K]) = p_c(d) + \eta > p_c(d)$, we can choose $M=K$ in the definition of the times $\wT$. We recall that $|A|$ denotes the cardinality of a discrete set $A$. We have
\begin{align}
\label{eqC(0)}
\EE \left[ |C_0(0)| \right] & \,=\, \EE [|C_0(0)| \ind{\{ 0 \in \C_M \}} ] + \EE[|C_0(0)| \ind{\{ 0 \notin \C_M \}} ] \nonumber \\
& \,\leq \,  \EE [|C_0(0)| \ind{\{ 0 \in \C_M \}} ] + \EE[|C_M(0)| \ind{\{ 0 \notin \C_M \}} ] \,.
\end{align}
To control the last term in (\ref{eqC(0)}), we can use Theorem \ref{thmfinite} since the percolation $(\ind{\{ t(e)\leq M \}}, e\in \EE^d)$ is supercritical :
\begin{align}
\label{eqC(0)2}
\EE[|C_M(0)| \ind{\{ 0 \notin \C_M \}} ] & \,\leq\, \sum_{k\in \NN} (k+1) \, \PP[k\leq |C_M(0)| < k+1\,,\, 0 \notin \C_M ]  \nonumber \\
& \,\leq \, \sum_{k\in \NN} \PP [|C_M(0)| \geq k\,,\, 0 \notin \C_M] \nonumber \\
& \,\leq\, \sum_{k\in \NN} \PP[ 0\notin \C_M\,,\, 0 \overset{M}{\longleftrightarrow} \partial B_1 (0,k) ] \nonumber \\
& \,\leq \, \sum_{k\in \NN} A_4 e^{-A_5 k} \,<\, +\infty \,.
\end{align}
We recall that $H_n = \{ z \in \RR^d \,|\, z_1 = n\}$, and $T(0,H_n) = \inf \{ T(0,x) \,|\, x \in H_n \cap \ZZ^d \} $. Concerning the first term in (\ref{eqC(0)}), we have similarly
$$\EE [|C_0(0)| \ind{\{ 0 \in \C_M \}} ] \,\leq\, \sum_{k\in \NN} \PP [|C_0(0)| \geq k\,,\, 0 \in \C_M] \,.$$
Notice that if $C_0(0) \subset [-n, n]^d$ for $n\in \NN$, then $|C_0(0)| \leq (2n+1)^d$. This means that if $|C_0(0) |\geq k$, then $C_0(0) \not\subset [-(k^{1/d} -1 )/2, (k^{1/d}-1)/2]^d$. Using the symmetry of the model we obtain
\begin{align*}
\EE [|C_0(0)| \ind{\{ 0 \in \C_M \}} ]  &\,\leq\, \sum_{k\in \NN} \PP \left[ \{ 0 \in \C_M \} \cap \left\{ \begin{array}{c}\exists i\in \{1,\dots ,d\}\,,\, \exists x\in \ZZ^d \textrm{ s.t. } \\ T(0, x) =0 \textrm{ and } |x_i| \geq (k^{1/d}-1)/2 \end{array}\right\} \right]\\
& \,\leq\, \sum_{k\in \NN} 2d \,\, \PP[0\in \C_M \,,\, T(0,H_{(k^{1/d}-1)/2})=0] \,.
\end{align*}
Suppose that $\wmu(e_1) > 0$. Then by Proposition \ref{propGD} we conclude that
\begin{equation}
\label{eqC(0)3}
\EE [|C_0(0)| \ind{\{ 0 \in \C_M \}} ]   \,\leq \, 2d \sum_{k\in \NN}  C_5 \exp \left[ -\frac{C_6}{2}  (k^{1/d}-1) \right] \,<\, +\infty \,.
\end{equation}
Thus, if $\wmu(e_1)$ if finite, we obtain by (\ref{eqC(0)}), (\ref{eqC(0)2}) and (\ref{eqC(0)3}) that $\EE [|C_0(0)|] < \infty$. But this implies that $F(\{ 0\}) < p_c(d)$ (see for instance Corollary 5.1 in \cite{Kesten:perco}), which is a contradiction. Thus $\wmu(e_1) = 0$.
\end{proof}

\begin{proof}[Proof of Theorem \ref{thmpositivity}]
Let $F$ satisfying $F([0,+\infty[)>p_c(d)$. If $F(\{ 0 \}) <p_c(d)$, we conclude by Proposition \ref{propF(0)<pc}. If $F(\{ 0 \}) >p_c(d)$, we conclude by Proposition \ref{propF(0)>pc}. Suppose that $F(\{ 0 \}) = p_c(d)$. Since $F([0,+\infty[)>p_c(d)$, there exists $K \in ]0,+\infty[$ such that $F([0,K])>p_c(d)$. We define 
$$ t'(e) \,=\, \left\{ \begin{array}{ll} 0 & \textrm{ if }t(e)=0\\ K & \textrm{ if }t(e) \in ]0,K] \\ +\infty & \textrm{ if } t(e) >K \,. \end{array} \right. $$
By construction $t'(e) \geq t(e) $ for all $e\in \EE^d$. Choosing $M=K$ in the definition of the regularized times $\wT$ and $\wT'$ (corresponding to the times $t'(e)$), we see that the infinite cluster $\C_M$ of edges of passage time smaller than $M=K$ is exactly the same according to the times $t$ and $t'$, thus the points $\wx$ are the same, and we obtain that $\wT(x,y) \leq \wT'(x,y)$, for all $x,y\in \ZZ^d$. If $\wmu'(e_1)$ denotes the time constant for the passage times $t'(e)$, we conclude that $\wmu(e_1) \leq \wmu'(e_1)$. Applying Proposition \ref{propF(0)=pc}, we obtain that $\wmu'(e_1)=0$, thus $\wmu(e_1)=0$.
\end{proof}


\section*{Acknowledgements} 

The authors would like to thank Daniel Boivin for bringing to their attention the reference \cite{Mourrat12}.

\def\cprime{$'$}

\end{document}